\RequirePackage{etex}
\documentclass[pdftex,a4paper]{article}
\usepackage{amsmath,amssymb,amsthm}
\usepackage{microtype}
\usepackage{tikz}
\usetikzlibrary{arrows,cd}
\usepackage[hypertexnames=false]{hyperref}
\hypersetup{
	colorlinks=true,
	citecolor=red,
	linkcolor=blue,
	urlcolor=orange
	}
\usepackage{cleveref}
\usepackage[square,numbers]{natbib}
\usepackage{autonum}
\usepackage{enumitem}
\setlist[enumerate,1]{label=(\arabic*),ref=(\arabic*)}
\setlist[enumerate,2]{label=(\alph*),ref=\theenumi. (\alph*)}
\crefname{equation}{equation.}{equations.}

\numberwithin{equation}{subsection}
\theoremstyle{definition}

\newtheorem{theorem}{Theorem}[section]
\crefname{theorem}{Theorem}{Theorems}
\newtheorem{maintheorem}[theorem]{Main Theorem}
\crefname{maintheorem}{Main Theorem}{Main Theorems}
\newtheorem{lemma}[theorem]{Lemma}
\crefname{lemma}{Lemma}{Lemmas}
\newtheorem{proposition}[theorem]{Proposition}
\crefname{proposition}{Proposition}{Propositions}
\newtheorem{corollary}[theorem]{Corollary}
\crefname{corollary}{Corollary}{Corollaries}
\newtheorem{definition}[theorem]{Definition}
\crefname{definition}{Definition}{Definitions}
\newtheorem{example}[theorem]{Example}
\crefname{example}{Example}{Examples}
\newtheorem{remark}[theorem]{Remark}
\crefname{remark}{Remark}{Remarks}

\newcommand{\setmid}{\; \middle|\;}
\newcommand{\lmod}{\operatorname{\mathrm{\hspace{-2pt}-mod}}}
\newcommand{\lMod}{\operatorname{\mathrm{\hspace{-2pt}-Mod}}}
\newcommand{\lproj}{\operatorname{\mathrm{\hspace{-2pt}-proj}}}

\newcommand{\thick}{\operatorname{\mathrm{thick}}}
\newcommand{\silt}{\operatorname{\mathrm{silt}}}
\newcommand{\tilt}{\operatorname{\mathrm{tilt}}}
\newcommand{\opo}{\operatorname{\mathrm{op}}}
\newcommand{\twosilt}{\operatorname{\mathrm{2-silt}}}
\newcommand{\twotilt}{\operatorname{\mathrm{2-tilt}}}

\newcommand{\Image}{\operatorname{\mathrm{Im}}}
\newcommand{\Kernel}{\operatorname{\mathrm{Ker}}}
\newcommand{\Soc}{{\operatorname{Soc}\nolimits}}
\newcommand{\Rad}{{\operatorname{Rad}\nolimits}}
\newcommand{\Ext}{\operatorname{Ext}\nolimits}
\newcommand{\End}{\operatorname{End}\nolimits}
\newcommand{\Hom}{\operatorname{Hom}\nolimits}

\newcommand{\induc}{{\operatorname{Ind}\nolimits}}
\newcommand{\restr}{{\operatorname{Res}\nolimits}}
\newcommand{\add}{\operatorname{\mathrm{add}}}
\newcommand{\stautilt}{\operatorname{\mathrm{s\tau-tilt}}}
\newcommand{\itaurigid}{\operatorname{\mathrm{indec.\tau-rigid}}}
\newcommand{\itauirigid}{\operatorname{\mathrm{indec.\tau^{-1}-rigid}}}
\newcommand{\stauitilt}{\operatorname{\mathrm{s\tau^{-1}-tilt}}}
\newcommand{\nscf}{\operatorname{\mathrm{s}}}

\newcommand{\tors}{\operatorname{\mathrm{tors}}}
\newcommand{\torf}{\operatorname{\mathrm{torf}}}
\newcommand{\ftors}{\operatorname{\mathrm{f-tors}}}
\newcommand{\ftorf}{\operatorname{\mathrm{f-torf}}}
\newcommand{\Fac}{\operatorname{\mathrm{Fac}}}
\newcommand{\Sub}{\operatorname{\mathrm{Sub}}}
\newcommand{\Filt}{\operatorname{\mathrm{Filt}}}
\newcommand{\brick}{\operatorname{\mathrm{brick}}}
\newcommand{\sbrick}{\operatorname{\mathrm{sbrick}}}
\newcommand{\flsbrick}{\operatorname{\mathrm{f_L-sbrick}}}
\newcommand{\flbrick}{\operatorname{\mathrm{f_L-brick}}}
\newcommand{\frsbrick}{\operatorname{\mathrm{f_R-sbrick}}}
\newcommand{\frbrick}{\operatorname{\mathrm{f_R-brick}}}
\newcommand{\smc}{\operatorname{\mathrm{smc}}}
\newcommand{\twosmc}{\operatorname{\mathrm{2-smc}}}
\newcommand{\torscl}{\operatorname{\mathrm{T}}}
\newcommand{\torfcl}{\operatorname{\mathrm{F}}}
\newcommand{\inertiagp}{I}
\newcommand{\decompgp}{I}
\newcommand{\Id}{\mathrm{Id}}

\newcommand{\Addresses}{{
  \bigskip
  \footnotesize

  Ryotaro~KOSHIO (Corresponding author)\par\nopagebreak
  \textsc{Department of Mathematics, Tokyo University of Science}
  \par\nopagebreak
	1-3, Kagurazaka, Shinjuku-ku, Tokyo, 162-8601, Japan
  \par\nopagebreak
  E-mail: \href{mailto:1120702@ed.tus.ac.jp}{1120702@ed.tus.ac.jp}

  \medskip

  Yuta~KOZAKAI
  \par\nopagebreak
  \textsc{Department of Mathematics, Tokyo University of Science}
	\par\nopagebreak
	1-3, Kagurazaka, Shinjuku-ku, Tokyo, 162-8601, Japan
	\par\nopagebreak
  E-mail: \href{mailto:kozakai@rs.tus.ac.jp}{kozakai@rs.tus.ac.jp}
}}
\title{Induced modules of support \(\tau\)-tilting modules and extending modules of semibricks over blocks of finite groups
}

\author{Ryotaro~KOSHIO \and Yuta~KOZAKAI}
\date{\today}
\begin{document}
\maketitle
\begin{abstract}
	In this article we study support \(\tau\)-tilting modules, semibricks and more over blocks of group algebras.
	Let \(k\) be an algebraically closed field of characteristic \(p>0\), \(\tilde{G}\) a finite group and \(G\) a normal subgroup of \(\tilde{G}\).
	Moreover, let \(\tilde{B}\) be a block of \(k\tilde{G}\) and \(B\) a block of \(kG\) covered by \(\tilde{B}\).
	We show that, under certain conditions for the factor group \(\tilde{G}/G\) and \(B\), induced modules and extending modules of support \(\tau\)-tilting modules and semibricks over \(B\) are also the ones over \(\tilde{B}\), respectively.
\end{abstract}

\tableofcontents

\section{Introduction}\label{intro 2021-11-04 11:17:58}
The study of derived equivalences of blocks of finite groups has been motived and inspired by ``Brou\'{e}'s conjecture'', which can be conceived of as a local-global principle in the modular representation theory of finite groups.
In \cite{MR1002456}, the solution to the problem of determining the equivalence of two given algebras was reduced to the problem of finding an appropriate tilting complex.
Therefore, abundant constructions of tilting complexes over blocks enables us to find the algebras which are derived equivalent to the blocks.
Of course, it is very hard to construct appropriate tilting complexes over blocks and to determine all tilting complexes over blocks.
The classes of tilting complexes called two-term tilting complexes are considered to be non-trivial and a bit easier to handle
because it is showed that there exists a one-to-one correspondence between the two-term tilting complexes and the support \(\tau\)-tilting modules over symmetric algebras in \cite{MR3187626}.
Abundant constructions of two-term tilting complexes over blocks are also useful for plenty of constructions of general tilting complexes over blocks by using the tilting mutations introduced in \cite{MR2927802}.
Therefore, we focus on support \(\tau\)-tilting modules, over blocks and their corresponding representation-theoretic objects, such as semibricks, functorially finite torsion classes of module categories, two-term simple-minded collections and more, which are also useful to study of derived equivalences of blocks \cite{MR4139031,MR3910476}.
Finally, we got some results which work effectively for the purpose stated above.

In order to describe these, we set notation as follows.
Let \(k\) be an algebraically closed field of characteristic \(p>0\), \(\tilde{G}\) a finite group, \(G\) a normal subgroup of \(\tilde{G}\), \(B\) a block of \(kG\) and \(\tilde{B}\) a block of \(k\tilde{G}\) covering \(B\), that is, \(1_B1_{\tilde{B}}\neq 0\), where \(1_B\) and \(1_{\tilde{B}}\) mean the respective unit elements of \(B\) and \(\tilde{B}\).
In this setting, there are some useful properties about the restriction functor \(\restr_G^{\tilde{G}}\) and the induction functor \(\induc_G^{\tilde{G}}\) between the category of \(B\)-modules and the one of \(\tilde{B}\)-modules.
We denote the inertial group of the block \(B\) in \({\tilde G}\) by \(\inertiagp_{\tilde G}(B)\)
and the second group cohomology of the factor group \(\inertiagp_{\tilde{G}}(B)/G\) with coefficients in the unit group \(k^\times\) of the field \(k\) with trivial \(G\)-action by \(H^2(\inertiagp_{\tilde{G}}(B)/G,k^\times)\).
We say that a \(B\)-module \(U\) is \(\inertiagp_{\tilde G}(B)\)-stable if \(xU \cong U\) as \(B\)-modules for any \(x\in \inertiagp_{\tilde G}(B)\).
Furthermore, we use the following notation:
\begin{itemize}
	\item \(\stautilt B\) (or \(\stautilt \tilde{B}\)) means the set of isomorphism classes of basic support \(\tau\)-tilting modules over \(B\) (or \(\tilde{B}\), respectively),
	\item \(\flsbrick B\) (or \(\flsbrick \tilde{B}\)) means the set of isomorphism classes of basic left finite semibricks over \(B\) (or \(\tilde{B}\), respectively),
	\item \(\twotilt B\) (or \(\twotilt \tilde{B}\)) means the set of isomorphism classes of basic two-term tilting complexes in \(K^b(B\lproj)\) (or \(K^b(\tilde{B}\lproj)\), respectively),
	\item \(\twosmc B\) (or \(\twosmc \tilde{B}\)) means the set of isomorphism classes of two-term simple-minded collections in \(D^b(B\lmod)\) (or \(D^b(\tilde{B}\lmod)\), respectively).
\end{itemize}
The following result contributes to abundant constructions of support \(\tau\)-tilting \(\tilde{B}\)-modules and two-term tilting complexes in \(K^b(\tilde{B}\lproj)\).
\begin{maintheorem}[{see \cref{MT 2021-09-07 13:50:56}}]\label{MT 2021-09-05 21:34:22}
	Under the above notation, we assume the following conditions hold:
	\begin{enumerate}
		\item Any left finite brick in \(B\lmod\) is \(\inertiagp_{\tilde G}(B)\)-stable.
		\item \(H^2(\inertiagp_{\tilde{G}}(B)/G,k^\times)=1\).
		\item \(k[\inertiagp_{\tilde{G}}(B)/G]\) is basic as a \(k\)-algebra.
	\end{enumerate}
	Then the maps
	\begin{equation}\label{stau corr 2021-09-05 21:41:23}
		\begin{tikzcd}[ampersand replacement=\&,row sep=1pt]
			\stautilt B \ar[r]\&\stautilt \tilde{B}\\
		\end{tikzcd}
	\end{equation}
	defined by \(\stautilt B\ni M\mapsto \tilde{B}\induc_G^{\tilde{G}} M \in \stautilt \tilde{B}\) and
	\begin{equation}\label{twotilt corr 2021-09-05 21:41:23}
		\begin{tikzcd}[ampersand replacement=\&,row sep=1pt]
			\twotilt B \ar[r]\&\twotilt \tilde{B}\\
		\end{tikzcd}
	\end{equation}
	defined by \(\twotilt B\ni T\mapsto \tilde{B}\induc_G^{\tilde{G}} T \in \twotilt \tilde{B}\) are well-defined and injective.
	Moreover, we get the following commutative diagram:
	\begin{equation}
		\begin{tikzcd}[ampersand replacement=\&, column sep=2cm]
			\stautilt B\ar[d,"\text{\cite{MR3187626} for \(B\)}"',"\wr"] \ar[r,"\text{\eqref{stau corr 2021-09-05 21:41:23}}"]\&\stautilt \tilde{B} \ar[d,"\text{\cite{MR3187626} for \(\tilde{B}\)}","\wr"']\\
			\twotilt B\ar[r,"\text{\eqref{twotilt corr 2021-09-05 21:41:23}}"']\&\twotilt \tilde{B}.
		\end{tikzcd}
	\end{equation}
\end{maintheorem}

One of our interests is, for finite group \(\tilde{G}\) and its normal subgroup \(G\) with a cyclic Sylow \(p\)-subgroup and of \(p\)-power index in \(\tilde{G}\), whether the representation theoretical objects over \(\tilde{B}\) can be obtained from those over \(B\), where \(B\) is a block of \(kG\) and \(\tilde{B}\) is that of \(k\tilde{G}\) covering \(B\) (for example see \cite{MR1889341,MR2592757}).
In \cite{MR4243358}, the authors showed that the induction functor \(\induc_G^{\tilde{G}}\) induces a poset isomorphism between \(\stautilt B\) and \(\stautilt \tilde{B}\) and between \(\twotilt B\) and \(\twotilt \tilde{B}\) in the above setting.
On the other hand, there must be an explicit correspondence between \(\flsbrick B\) and \(\flsbrick \tilde{B}\) and that between \(\twosmc B\) and \(\twosmc \tilde{B}\), but it was not made clear that how they correspond in that paper.
Therefore, one of our goals is clarifying the correspondences, and we get the following theorems as positive answers which can be applied to our interested situation.

\begin{maintheorem}[{see \cref{MT sbrick 2021-09-10 09:27:04}}]\label{MT sbrick 2021-09-06 11:12:51}
	With the same assumptions in \cref{MT 2021-09-05 21:34:22}, the following hold:
	\begin{enumerate}
		\item Let \(e\) be the number of isomorphism classes of simple \(k[\inertiagp_{\tilde{G}}(B)/G]\)-modules. Then for any left finite semibrick \(S\) in \(B\lmod\) and any indecomposable direct summand \(S_i\) of \(S\), there exist exactly \(e\) isomorphism classes of bricks \(\tilde{S_i}^{(1)}, \ldots, \tilde{S_i}^{(e)}\) in \(k\inertiagp_{\tilde{G}}(B)\lmod\) satisfying \(\restr_G^{\inertiagp_{\tilde{G}}(B)}\tilde{S}_i^{(j)}\cong S_i\) for all \(i,j\).
		\item The map
		      \begin{equation}\label{sbrick corr 2021-09-05 21:41:10}
			      \begin{tikzcd}[ampersand replacement=\&,row sep=1pt]
				      \flsbrick B \ar[r]\&\flsbrick \tilde{B}\\
			      \end{tikzcd}
		      \end{equation}
		      defined by \(S\cong \bigoplus_{i=1}^{n_{S}}S_i\mapsto \tilde{B} \induc^{\tilde{G}}_{\inertiagp_{\tilde{G}}(B)}\left( \bigoplus_{i=1}^{n_S}\bigoplus_{j=1}^{e}\tilde{S_i}^{(j)} \right)\), here \(S\cong \bigoplus_{i=1}^{n_{S}}S_i\) is a direct sum decomposition into bricks, is well-defined and injective.
		\item We get the following commutative diagram:
		      \begin{equation}
			      \begin{tikzcd}[ampersand replacement=\&, column sep=2cm]
				      \stautilt B\ar[d,"\text{\cite{MR4139031} for \(B\)}"',"\wr"] \ar[r,"\text{\eqref{stau corr 2021-09-05 21:41:23}}"]\&\stautilt \tilde{B} \ar[d,"\text{\cite{MR4139031} for \(\tilde{B}\)}","\wr"']\\
				      \flsbrick B\ar[r,"\text{\eqref{sbrick corr 2021-09-05 21:41:10}}"']\&\flsbrick \tilde{B}.
			      \end{tikzcd}
		      \end{equation}
	\end{enumerate}
\end{maintheorem}
\noindent
In addition to the above theorem, in our setting, we get an explicit map
\begin{equation}\label{corr smc 2021-09-06 11:23:00}
	\begin{tikzcd}[ampersand replacement=\&,row sep=1pt]
		\twosmc B \ar[r]\&\twosmc \tilde{B}
	\end{tikzcd}
\end{equation}
induced by \eqref{sbrick corr 2021-09-05 21:41:10} (see \cref{MT smcsbrick 2021-10-03 21:01:38} for the detail) and get the following result.
\begin{maintheorem}[{see \cref{big com. diag. 2021-10-06 10:05:34}}]\label{MT 2022-08-24 20:58:25}
	With the same assumptions in \cref{MT 2021-09-05 21:34:22}, the following diagram is commutative:
	\begin{equation}\label{diagram 2022-08-24 21:38:01}
		\begin{tikzcd}
			&\twosmc B\ar[dd,"\text{\cite{MR4139031} for \(B\)}"'very near end]  \ar[rr,"\text{\eqref{corr smc 2021-09-06 11:23:00}}"]&&\twosmc \tilde{B}\\
			\twotilt B\ar[ru,"\text{\cite{MR4139031,MR3910476} for \(B\)}"description] \ar[rr,"\text{\eqref{twotilt corr 2021-09-05 21:41:23}}"near end,crossing over]&&\twotilt \tilde{B}\ar[ru,"\text{\cite{MR4139031,MR3910476} for \(\tilde{B}\)}"description] \\
			&\flsbrick B\ar[rr,"\eqref{sbrick corr 2021-09-05 21:41:10}"very near start]&&\flsbrick \tilde{B}\ar[uu,leftarrow,"\text{\cite{MR4139031} for \(\tilde{B}\)}"']\\
			\stautilt B\ar[ru,"\text{\cite{MR4139031} for B}"description] \ar[uu,"\text{\cite{MR3187626} for \(B\)}"] \ar[rr,"\text{\eqref{stau corr 2021-09-05 21:41:23}}"'] &&\stautilt \tilde{B} \ar[ru,"\text{\cite{MR4139031} for \(\tilde{B}\)}"description] \ar[uu,"\text{\cite{MR3187626} for \(\tilde{B}\)}"'very near end,crossing over].
		\end{tikzcd}
	\end{equation}
\end{maintheorem}

At a glance, the assumption in \cref{MT 2021-09-05 21:34:22}, which is also required in \cref{MT 2022-08-24 20:58:25,MT sbrick 2021-09-06 11:12:51}, seems strong, but in fact it can be applied to many situations, including the one we are interested in.
In the setting of following \cref{MT example 2022-08-24 21:40:52}, the conditions of \cref{MT 2021-09-05 21:34:22} are satisfied automatically.
In that sense, \cref{MT 2021-09-05 21:34:22} is a generalization of the main theorem in \cite{MR4243358}, and \cref{MT sbrick 2021-09-06 11:12:51,MT 2022-08-24 20:58:25} which require the conditions of \cref{MT 2021-09-05 21:34:22} bring more representation theoretical information on the covering blocks including the classes dealt with in \cite{MR4243358}.
\begin{maintheorem}[{see \cref{bijectice 2022-08-24 21:47:43,cyclic quot principal 2022-08-28 22:59:59,example 2022-08-24 21:48:34,example 2022-08-24 21:48:53}}]\label{MT example 2022-08-24 21:40:52}
	Let \(G\) be a normal subgroup of a finite group \(\tilde{G}\), \(B\) a block of \(kG\) and \(\tilde{B}\) a block of \(k\tilde{G}\) covering \(B\) satisfying one of the following conditions,
	then the assumptions of \cref{MT 2021-09-05 21:34:22} hold.
	In particular, in the situation \ref{2022-08-24 21:36:07}, the all horizontal maps in \eqref{diagram 2022-08-24 21:38:01} are bijective and all bricks over \(\tilde{B}\) can be obtained by the extensions of those over \(B\).
	\begin{enumerate}
		\item \(G\) has a cyclic Sylow \(p\)-subgroup and the quotient group \(\tilde{G}/G\) is a \(p\)-group.\label{2022-08-24 21:36:07}
		\item \(G\) has a cyclic Sylow \(p\)-subgroup, the quotient group \(\tilde{G}/G\) is a cyclic group or isomorphic to the dihedral group \(D_{2p}\) of order \(2p\) and \(B\) is the principal block \(B_0(kG)\) or a block of \(kG\) with distinct dimensional simple \(B\)-modules to each other.
		\item \(G=\mathfrak{A}_5\) (the alternating group of degree \(5\)) and \(\tilde{G}=\mathfrak{S}_5\) (the symmetric group of degree \(5\)) where \(p=5\).
		\item \(G\) is an arbitrary finite group and \(\tilde{G}=G\times H\), where \(H\) is a \(p\)-group, a cyclic group or the dihedral group of order \(2p\).
	\end{enumerate}
\end{maintheorem}

In this paper, we use the following notation and convention.
Modules mean finitely generated left modules and complexes mean cochain complexes.
Let \(\Lambda\) be a finite dimensional algebra over a field \(k\).
For a \(\Lambda\)-module \(U\), we denote by \(\Rad(U)\) the Jacobson radical of \(U\), by \(\Soc(U)\) the socle of \(U\), by \(P(U)\) the projective cover of \(U\), by \(I(U)\) the injective envelope of \(U\), by \(\Omega(U)\) the syzygy of \(U\), by \(\Omega^{-1}(U)\) the cosyzygy of \(U\) and by \(\tau U\) the Auslander--Reiten translate of \(U\).
We denote by \(\Lambda\lmod\) the module category of \(\Lambda\), by \(K^b(\Lambda\lproj)\) the homotopy category consisting of bounded complexes of projective \(\Lambda\)-modules and by \(D^b(\Lambda\lmod)\) the bounded derived category consisting of complexes of \(\Lambda\)-modules.
For an object \(X\) of \(\Lambda\lmod\) (of \(K^b(\Lambda\lproj)\), of \(D^b(\Lambda\lmod)\)),
we denote by \(\add X\) the full subcategory of \(\Lambda\lmod\) (of \(K^b(\Lambda\lproj)\), of \(D^b(\Lambda\lmod)\) respectively) whose objects are isomorphic to finite direct sums of direct summands of \(X\).
For \(\Lambda\)-modules \(U\) and \(U'\), we denote by \(\Rad_{\Lambda\lmod}(U,U')\) the Jacobson radical of \(\Hom_\Lambda(U,U')\).
We say that an object \(X\) of \(\Lambda\lmod\), \(K^b(\Lambda\lproj)\) or \(D^b(\Lambda\lmod)\) is basic if any two indecomposable direct summands of \(X\) are non-isomorphic.
We denote by \(\nu_\Lambda\) the Nakayama functor of \(\Lambda\lmod\) which maps any projective \(\Lambda\)-module to injective \(\Lambda\)-module.

\section{Preliminary results of \texorpdfstring{\(\tau\)}{tau}-tilting theory}
In this section, \(k\) means an algebraically closed field and \(\Lambda\) means a finite dimensional \(k\)-algebra.

\subsection{Functorially finiteness of torsion classes and torsion-free classes}\label{pre tors 2021-11-04 11:32:08}
Let \(\mathcal{C}\) be a full subcategory of the module category \(\Lambda \lmod\).
We say that \(\mathcal{C}\) is contravariantly finite in \(\Lambda \lmod\) if any object in \(\Lambda\lmod\) has a right \(\mathcal{C}\)-approximation, that is, for every object \(M\) of \(\Lambda \lmod\) there exist an object \(C\) of \(\mathcal{C}\) and a morphism \(f\colon C\rightarrow M\) such that the sequence of functors from \(\mathcal{C}\) to \(k\lmod\)
\begin{equation}
	\begin{tikzcd}
		\Hom_{\Lambda}(-,C)|_{\mathcal{C}} \ar[r,"f \circ \bullet"]& \Hom_{\Lambda}(-,M)|_{\mathcal{C}} \ar[r]& 0
	\end{tikzcd}
\end{equation}
is exact.
Dually, we say that \(\mathcal{C}\) is covariantly finite in \(\Lambda \lmod\) if any object in \(\Lambda\lmod\) has a left \(\mathcal{C}\)-approximation, that is, for every object \(M\) of \(\Lambda \lmod\) there exist an object \(C\) of \(\mathcal{C}\) and a morphism \(g\colon M\rightarrow C\) such that the sequence of functors from \(\mathcal{C}\) to \(k\lmod\)
\begin{equation}
	\begin{tikzcd}
		\Hom_{\Lambda}(C,-)|_{\mathcal{C}} \ar[r,"\bullet \circ g"]& \Hom_{\Lambda}(M,-)|_{\mathcal{C}} \ar[r]& 0
	\end{tikzcd}
\end{equation}
is exact.
We say that \(\mathcal{C}\) is functorially finite if \(\mathcal{C}\) is both contravariantly finite and covariantly finite in \(\Lambda\lmod\).
We denote the right perpendicular subcategory of \(\mathcal{C}\) by
\begin{equation}
	\mathcal{C}^{\perp}:=\left\{ X \in \Lambda \lmod \setmid \text{\(\forall C \in \mathcal{C}, \Hom_{\Lambda}(C, X) =0\)} \right\}
\end{equation}
and the left perpendicular subcategory of \(\mathcal{C}\) by
\begin{equation}
	{}^{\perp}\mathcal{C}:=\left\{ X \in \Lambda \lmod \setmid \forall C \in \mathcal{C}, \Hom_{\Lambda}(X,C) =0 \right\}.
\end{equation}
We denote by \(\Fac(\mathcal{C})\) the full subcategory of \(\Lambda\lmod\) consisting of all factor modules of finite direct sums of objects in \(\mathcal{C}\).
Dually, we denote by \(\Sub(\mathcal{C})\) the full subcategory of \(\Lambda\lmod\) consisting of all submodules of finite direct sums of objects in \(\mathcal{C}\). We denote by \(\Filt(\mathcal{C})\) the full subcategory of \(\Lambda\lmod\) consisting of all modules having a finite \(\add \mathcal{C}\)-filtration, that is,

\begin{equation}
	\Filt(\mathcal{C}):=\left\{M \in \Lambda\lmod \setmid \parbox{6cm}{there exist \(l\in\mathbb{N}\) and a sequence \\ \(0=M_0\subset M_1\cdots \subset M_{l-1}\subset M_l=M\) \\of \(\Lambda\)-modules with \(M_{i}/M_{i-1} \in \add \mathcal{C}\) for all \(i=1, \ldots, l\).}
	\right\}.
\end{equation}

Let \(\mathcal{T}\) and \(\mathcal{F}\) be full subcategories of the module category \(\Lambda\lmod\).
We say that \(\mathcal{T}\) is a torsion class if \(\mathcal{T}\) is closed under taking factor modules, direct sums and extensions.
Dually, we say that \(\mathcal{F}\) is a torsion-free class if \(\mathcal{F}\) is closed under taking submodules, direct sums and extensions.
We use the following notation:
\begin{itemize}
	\item \(\tors \Lambda\) means the set of torsion classes in \(\Lambda \lmod\),
	\item \(\torf \Lambda\) means the set of torsion-free classes in \(\Lambda \lmod\),
	\item \(\ftors \Lambda\) means the set of functorially finite torsion classes in \(\Lambda \lmod\),
	\item \(\ftorf \Lambda\) means the set of functorially finite torsion-free classes in \(\Lambda \lmod\).
\end{itemize}

These sets are ordered by inclusion.
Let \(\mathcal{C}\) be a full subcategory of the module category \(\Lambda \lmod\).
We define \(\torscl(\mathcal{C})\) (or \(\torfcl(\mathcal{C})\)) to be the smallest torsion class (or torsion-free class, respectively) containing \(\mathcal{C}\).
For a \(\Lambda\)-module \(U\),
we abbreviate \(\Fac(\add U)\), \(\Sub(\add U)\), \(\torscl (\add U)\) and \(\torfcl (\add U)\) as \(\Fac(U)\), \(\Sub(U)\), \(\torscl(U)\) and \(\torfcl(U)\), respectively.

The following assertion is obvious, but plays an important role, so we give its proof briefly.
\begin{proposition}[{\cite[Lemma 3.1]{MR3723626}}]\label{filt and closure 2021-12-06 11:12:23}
	Let \(\mathcal{C}\) be a full subcategory of the module category \(\Lambda \lmod\).
	Then we have \(\torscl(\mathcal{C})=\Filt(\Fac(\mathcal{C}))\) and \(\torfcl(\mathcal{C})=\Filt(\Sub(\mathcal{C}))\).
\end{proposition}
\begin{proof}
	We only prove that \(\torscl(\mathcal{C})=\Filt(\Fac(\mathcal{C}))\); the other statement \(\torfcl(\mathcal{C})=\Filt(\Sub(\mathcal{C}))\) follows similarly.
	First, we show that the subcategory \(\Filt(\Fac(\mathcal{C}))\) of \(\Lambda\lmod\) is a torsion class.
	It is obvious that \(\Filt(\Fac(\mathcal{C}))\) is closed under direct sums and extensions.
	Let \(U\) be an arbitrary object of \(\Filt(\Fac(\mathcal{C}))\) and \(U\xrightarrow{f} V\) an arbitrary epimorphism from \(U\) to a \(\Lambda\)-module \(V\).
	Then we can take a filtration
	\begin{equation}
		0=U_0\subset U_1\subset \cdots\subset U_{l-1}\subset U_l=U
	\end{equation}
	of \(U\) satisfying \(U_{i}/U_{i-1} \in \add \Fac(\mathcal{C})=\Fac(\mathcal{C})\) for all \(i=1, \ldots, l\). Then we have the filtration
	\begin{equation}
		0=f[U_0]\subset f[U_1]\subset \cdots\subset f[U_{l-1}]\subset f[U_l]=V
	\end{equation}
	of \(V\).
	For any \(i=1, \ldots, l\), the epimorphism \(f\) induces an epimorphism
	\begin{equation}
		U_i/U_{i-1}\rightarrow f[U_{i}]/f[U_{i-1}].
	\end{equation}
	Since \(\Fac(\mathcal{C})\) is closed under taking factor modules, we have \(f[U_{i}]/f[U_{i-1}] \in \Fac(\mathcal{C})\) and \(V\in \Filt(\Fac(\mathcal{C}))\).
	Hence, we have that \(\Filt(\Fac(\mathcal{C}))\) is closed under taking factor modules. Thus, we get that \(\Filt(\Fac(\mathcal{C}))\) is a torsion class in \(\Lambda\lmod\).

	Next we show that \(\Filt(\Fac(\mathcal{C}))=\torscl(\mathcal{C})\).
	Since \(\Filt(\Fac(\mathcal{C}))\) is a torsion class and \(\Filt(\Fac(\mathcal{C}))\supset \mathcal{C}\), we get \(\Filt(\Fac(\mathcal{C})) \supset \torscl(\mathcal{C})\) by minimality of \(\torscl(\mathcal{C})\).
	Moreover, the inclusion \(\mathcal{C}\subset \torscl(\mathcal{C})\) and the definitions of torsion classes imply that
	\(\Filt(\Fac(\mathcal{C}))\subset\torscl(\mathcal{C})\).
\end{proof}
The following proposition which gives the connection between torsion classes and torsion free classes is crucial.

\begin{proposition}[{for example, see \cite{MR2197389}}]\label{dualtorsf 2021-12-16 16:31:45}
	The following maps are mutually inverse isomorphisms of partially ordered sets:
	\begin{equation}
		\begin{tikzcd}[row sep=1pt]
			\tors \Lambda \ar[r]& (\torf \Lambda)^{\opo}\\
			\mathcal{T}\ar[r,mapsto]& {}^{\perp}\mathcal{T},
		\end{tikzcd}
	\end{equation}
	\begin{equation}
		\begin{tikzcd}[row sep=1pt]
			\torf \Lambda \ar[r]& (\tors \Lambda)^{\opo}\\
			\mathcal{F}\ar[r,mapsto]& \mathcal{F}^{\perp}.
		\end{tikzcd}
	\end{equation}
	Moreover, these isomorphisms restrict to the following isomorphisms of partially ordered sets, respectively:
	\begin{equation}\label{gal}
		\begin{tikzcd}[row sep=1pt]
			\ftors \Lambda \ar[r]&(\ftorf \Lambda)^{\opo}\\
			\mathcal{T}\ar[r,mapsto]& {}^{\perp} \mathcal{T},
		\end{tikzcd}
	\end{equation}
	\begin{equation}\label{gali}
		\begin{tikzcd}[row sep=1pt]
			\ftorf \Lambda \ar[r]&(\ftors \Lambda)^{\opo}\\
			\mathcal{F}\ar[r,mapsto]& \mathcal{F}^{\perp}.
		\end{tikzcd}
	\end{equation}
\end{proposition}

\subsection{Support \texorpdfstring{\(\tau\)}{tau}-tilting modules}
We recall the definitions and basic properties of support \(\tau\)-tilting modules and support \(\tau^{-1}\)-tilting modules which are dual notion of support \(\tau\)-tilting modules.
For a \(\Lambda\)-module \(M\), we denote by \(|M|\) the number of isomorphism classes of indecomposable direct summands of \(M\).
In particular, \(|\Lambda|:=|{}_\Lambda \Lambda |\) means the number of isomorphism classes of simple \(\Lambda\)-modules.
Also, we denote by \(\nscf(M)\) the number of isomorphism classes of simple modules appearing as composition factors of \(M\).
\begin{definition}[{\cite[Definition 0.1]{MR3187626}}]
	Let \(M\) be a \(\Lambda\)-module.
	\begin{enumerate}
		\item We say that \(M\) is \(\tau\)-rigid if \(\Hom_{\Lambda}(M,\tau M)=0\).
		\item We say that \(M\) is \(\tau\)-tilting if \(M\) is a \(\tau\)-rigid module and \(|M|=|\Lambda|.\)
		\item We say that \(M\) is support \(\tau\)-tilting if there exists an idempotent \(e\) of \(\Lambda\) such that \(M\) is a \(\tau\)-tilting \(\Lambda/\Lambda e\Lambda \)-module.
	\end{enumerate}
\end{definition}

\begin{definition}[{The dual of \cite[Definition 0.1]{MR3187626}}]
	Let \(N\) be a \(\Lambda\)-module.
	\begin{enumerate}
		\item We say that \(N\) is \(\tau^{-1}\)-rigid if \(\Hom_{\Lambda}(\tau^{-1}N,N)=0.\)
		\item We say that \(N\) is \(\tau^{-1}\)-tilting if \(N\) is a \(\tau^{-1}\)-rigid module and \(|N|=|\Lambda|.\)
		\item We say that \(N\) is support \(\tau^{-1}\)-tilting if there exists an idempotent \(e\) of \(\Lambda\) such that \(N\) is a \(\tau^{-1}\)-tilting \(\Lambda/\Lambda e\Lambda \)-module.
	\end{enumerate}
\end{definition}

\begin{remark}[{\cite[Proposition 2.3 (a), (b)]{MR3848421}, \cite[Proposition 1.8]{MR3461065}}]\label{tau number remark 2021-09-07 12:11:53}
	Since \(e=0\) is an idempotent of \(\Lambda\) and \(\Lambda/\Lambda e\Lambda =\Lambda\), any \(\tau\)-tilting module (or any \(\tau^{-1}\)-tilting module) is a support \(\tau\)-tilting module (or a support \(\tau^{-1}\)-tilting module, respectively).
	Moreover, for any \(\tau\)-rigid \(\Lambda\)-module \(M\), the following conditions are equivalent:
	\begin{enumerate}
		\item \(M\) is support \(\tau\)-tilting module.
		\item There exist a projective \(\Lambda\)-module \(P\) satisfying that \(\Hom_\Lambda(P,M)=0\) and that \(|M|+|P|=|\Lambda|\).
		\item \(|M|=s(M)\).
	\end{enumerate}
\end{remark}

We use the following notation:
\begin{itemize}
	\item \(\stautilt \Lambda\) means the set of isomorphism classes of basic support \(\tau\)-tilting \(\Lambda\)-modules,
	\item \(\stauitilt \Lambda\) means the set of isomorphism classes of basic support \(\tau^{-1}\)-tilting \(\Lambda\)-modules,
	\item \(\itaurigid \Lambda\) means the set of isomorphism classes of indecomposable \(\tau\)-rigid \(\Lambda\)-modules,
	\item  \(\itauirigid \Lambda\) means the set of isomorphism classes of indecomposable \(\tau^{-1}\)-rigid \(\Lambda\)-modules.
\end{itemize}
\begin{proposition}[{\cite[Theorem 2.7, Theorem 2.15]{MR3187626}}]\label{2020/08/31}
	With the above notation, the following maps give bijections:
	\begin{equation}\label{fac}
		\begin{tikzcd}[row sep=1pt]
			\stautilt \Lambda \ar[r] &\ftors \Lambda\\
			M \ar[r,mapsto]& \Fac M,
		\end{tikzcd}
	\end{equation}
	\begin{equation}\label{sub}
		\begin{tikzcd}[row sep=1pt]
			\stauitilt \Lambda \ar[r] &\ftorf \Lambda\\
			N \ar[r,mapsto]& \Sub N.
		\end{tikzcd}
	\end{equation}
\end{proposition}
We can give \(\stautilt \Lambda\) and \(\stauitilt \Lambda\) partially ordered set structures by the above bijections and inclusions of \(\ftors \Lambda\) and \(\ftors \Lambda\).

\begin{definition}
	For \(M, M' \in \stautilt \Lambda\), we write \(M\geq M'\) if \(\Fac M\supset \Fac M'\), or equivalently there exist a positive integer \(r\) and an epimorphism
	\begin{equation}
		\begin{tikzcd}
			M^{\oplus r}\ar[r,twoheadrightarrow,"\varphi"]& M'.
		\end{tikzcd}
	\end{equation}
	Dually, for \(N', N \in \stauitilt \Lambda\), we write \(N'\leq N\) if \(\Sub N' \subset \Sub N\), or equivalently there exist a positive integer \(r\) and a monomorphism
	\begin{equation}
		\begin{tikzcd}
			N' \ar[r,hookrightarrow,"\psi"]& N^{\oplus r}.
		\end{tikzcd}
	\end{equation}
\end{definition}

Based on \cite[Theorem 2.33]{MR3187626}, we define support \(\tau\)-tilting mutations and support \(\tau^{-1}\)-tilting mutations as follows.

\begin{definition}[{\cite[Theorem 2.33]{MR3187626}}]
	Let \(M\) and \(M'\) be support \(\tau\)-tilting \(\Lambda\)-modules. We say that
	\(M'\) is a support \(\tau\)-tilting left mutation of \(M\) (or a support \(\tau\)-tilting right mutation) if \(M > M'\) holds and if there is no support \(\tau\)-tilting \(\Lambda\)-module \(L\) such that \(M>L>M'\) (or if \(M < M'\) holds and if there is no support \(\tau\)-tilting \(\Lambda\)-module \(L\) such that \(M<L<M'\), respectively).
\end{definition}

\begin{definition}[{Dual assetrion of \cite[Theorem 2.33]{MR3187626}}]
	Let \(N\) and \(N'\) be support \(\tau^{-1}\)-tilting \(\Lambda\)-modules. We say that
	\(N'\) is a support \(\tau^{-1}\)-tilting right mutation of \(N\) (or a support \(\tau^{-1}\)-tilting left mutation) if \(N < N'\) holds and if there is no support \(\tau^{-1}\)-tilting \(\Lambda\)-module \(L\) such that \(N<L<N'\) (or if \(N > N'\) holds and if there is no support \(\tau^{-1}\)-tilting \(\Lambda\)-module \(L\) such that \(N>L>N'\), respectively).
\end{definition}

We recall some fundamental properties of support \(\tau\)-tilting modules and support \(\tau^{-1}\)-tilting modules.

\begin{proposition}[{\cite[below Theorem 2.15 and Proposition 2.27]{MR3187626}}]\label{dugger 2021-09-06 13:33:14}
	The following hold:
	\begin{enumerate}
		\item The following map gives an isomorphism as partially ordered sets:
		      \begin{equation}\label{ddual}
			      \begin{tikzcd}[row sep=1pt]
				      \stautilt \Lambda \ar[r] &(\stauitilt \Lambda)^{\opo}\\
				      M \ar[r,mapsto]& \tau M\oplus \nu P,
			      \end{tikzcd}
		      \end{equation}
		      here \(P\) is a basic projective \(\Lambda\)-module satisfying that \(\Hom_\Lambda(P,M)=0\) and that \(|M|+|P|=|\Lambda|\).
		\item The above maps make the following diagram of partially ordered sets commutative:
		      \begin{equation}
			      \begin{tikzcd}
				      \stautilt \Lambda \ar[r,"\eqref{fac}"] \ar[d,"\eqref{ddual}"']&\ftors \Lambda\ar[d,"\eqref{gal}"]\\
				      (\stauitilt \Lambda)^{\opo} \ar[r,"\eqref{sub}"']&(\ftorf \Lambda)^{\opo}.
			      \end{tikzcd}
		      \end{equation}
	\end{enumerate}
\end{proposition}

\begin{proposition}[{\cite[Theorem 5.10]{MR617088}, \cite[Lemma 4.3, Lemma 4,4]{MR3910476}}]
	\label{indec tau taui corr2021-09-09 16:44:38}
	Let \(V\) be an indecomposable \(\Lambda\)-module. Then the following hold:
	\begin{enumerate}
		\item If \(V\) is a \(\tau\)-rigid module, then \(\Fac V\) is a torsion class of \(\Lambda\lmod\) and \(\tau V\) is a \(\tau^{-1}\)-rigid module.
		\item If \(V\) is a \(\tau^{-1}\)-rigid module, then \(\Sub V\) is a torsion-free class of \(\Lambda\lmod\) and \(\tau^{-1} V\) is a \(\tau\)-rigid module.
		\item The following map gives an injection:
		      \begin{equation}\label{2021-09-22 14:17:22}
			      \begin{tikzcd}[row sep=1pt]
				      \itaurigid \Lambda \ar[r] &\ftors \Lambda\\
				      X\ar[r,mapsto]&\Fac X.
			      \end{tikzcd}
		      \end{equation}
		\item The following map gives a bijection:
		      \begin{equation}\label{tau or top 2021-09-11 19:19:09}
			      \begin{tikzcd}[ampersand replacement=\&,row sep=1pt]
				      \itaurigid \Lambda \ar[r] \&\itauirigid \Lambda\\
				      X\ar[r,mapsto]\&\begin{cases}
					      \tau X & \text{(if \(X\) is non-projective),} \\
					      \nu X  & \text{(if \(X\) is projective)}.
				      \end{cases}
			      \end{tikzcd}
		      \end{equation}
		\item The following map is well-defined and injective:
		      \begin{equation}\label{2021-09-22 15:21:57}
			      \begin{tikzcd}[row sep=1pt]
				      \itaurigid \Lambda \ar[r]&\stautilt \Lambda\\
				      X\ar[r,mapsto] & M_X,
			      \end{tikzcd}
		      \end{equation}
		      here \(M_X\) is the support \(\tau\)-tilting \(\Lambda\)-module satisfying \(\Fac M_X=\Fac X\).
		\item The following map gives an injection:
		      \begin{equation}\label{2021-09-24 13:13:18}
			      \begin{tikzcd}[row sep=1pt]
				      \itauirigid \Lambda \ar[r] &\ftorf \Lambda\\
				      Y \ar[r,mapsto]& \Sub Y.
			      \end{tikzcd}
		      \end{equation}
		\item The following map is well-defined and injective:
		      \begin{equation}\label{2021-09-24 13:15:11}
			      \begin{tikzcd}[row sep=1pt]
				      \itauirigid \Lambda \ar[r]&\stauitilt \Lambda\\
				      Y\ar[r,mapsto]&N_Y,
			      \end{tikzcd}
		      \end{equation}
		      here \(N_Y\) is the support \(\tau^{-1}\)-tilting \(\Lambda\)-module satisfying \(\Sub N_Y=\Sub Y\).
		\item The above maps make the following diagram commutative:
		      \begin{equation}
			      \begin{tikzcd}
				      \stautilt \Lambda \ar[rd,leftarrow,"\eqref{2021-09-22 15:21:57}"]\ar[rrd,"\eqref{fac}",bend left=20pt]\ar[ddd,"\eqref{ddual}"]&&\\
				      &\itaurigid \Lambda \ar[r,"\eqref{2021-09-22 14:17:22}"] \ar[d,"\eqref{tau or top 2021-09-11 19:19:09}"']&\ftors \Lambda\ar[d,"\eqref{gal}"]\\
				      &\itauirigid \Lambda \ar[r,"\eqref{2021-09-24 13:13:18}"']&\ftorf \Lambda.\\
				      \stauitilt \Lambda\ar[ur,leftarrow,"\eqref{2021-09-24 13:15:11}"']\ar[urr,"\eqref{sub}"',bend right=20pt]&&
			      \end{tikzcd}
		      \end{equation}
	\end{enumerate}
\end{proposition}

\subsection{Bricks and semibricks}
We recall the definitions and basic properties of bricks and semibricks.
\begin{definition}
	Let \(S\) be \(\Lambda\)-module.
	\begin{enumerate}
		\item We say that a module \(S\) is a brick in \(\Lambda\lmod\) if \(\End_\Lambda(S)\cong k\).
		\item We say that a module \(S\) is a semibrick in \(\Lambda\lmod\)  if \(S\) is isomorphic to a direct sum of bricks \(S_1, \ldots, S_l\) in \(\Lambda\lmod\) which satisfy that \(\Hom_\Lambda(S_i,S_j)=0\) if \(S_i \ncong S_j\).
	\end{enumerate}
\end{definition}
\begin{definition}
	We say that a semibrick \(S\) in \(\Lambda\lmod\) is left finite (or right finite) if the torsion class \(\torscl(S)\), which is the smallest torsion class containing \(S\) (or if the torsion-free class \(\torfcl(S)\), which is the smallest torsion-free class containing \(S\), respectively), is functorially finite.
\end{definition}
We use the following notation:
\begin{itemize}
	\item \(\sbrick \Lambda\) means the set of isomorphism classes of basic semibricks in \(\Lambda \lmod\),
	\item \(\brick \Lambda\) means the set of isomorphism classes of bricks in \(\Lambda \lmod\),
	\item \(\flsbrick \Lambda\) means the set of isomorphism classes of basic left finite semibricks in \(\Lambda\lmod\),
	\item \(\frsbrick \Lambda\) means the set of isomorphism classes of basic right finite semibricks in \(\Lambda\lmod\),
	\item \(\flbrick \Lambda\) means the set of isomorphism classes of left finite bricks in \(\Lambda\lmod\),
	\item \(\frbrick \Lambda\) means the set of isomorphism classes of right finite bricks in \(\Lambda\lmod\).
\end{itemize}
For \(\Lambda\)-modules \(U\) and \(V\), we denote by \(R(U,V)\) the following submodule of \(V\):
\begin{equation}
	\sum_{f\in \Rad_{\Lambda\lmod}(U,V)}\Image f.
\end{equation}
We denote by \(S(U,V)\) the following submodule of \(U\):
\begin{equation}
	\bigcap_{f\in \Rad_{\Lambda\lmod}(U,V)}\Kernel f.
\end{equation}
\begin{remark}\label{serre simple 2021-09-09 19:04:50}
	It is easy to check that any semisimple module is a left finite semibrick and a right finite semibrick.
\end{remark}
\begin{theorem}
	[{\cite[Lemma 2.5, Proposition 2.13]{MR4139031}, \cite[Theorem 4.1]{MR3910476}}]\label{2021-11-22 19:52:15}
	Let \(M\) be a basic support \(\tau\)-tilting \(\Lambda\)-module and \(X\) an indecomposable direct summand of \(M\).
	Then the following hold:
	\begin{enumerate}
		\item The module \(X/R(X,X)\) is a left finite brick over \(\Lambda\). 
		\item The module \(X/R(M,X)\) is a brick or zero module. 
		\item The following conditions are equivalent:
		      \begin{enumerate}
			      \item The module \(X/R(M,X)\) is nonzero.
			      \item The module \(X/R(M,X)\) is a brick.
			      \item \(X\notin \Fac(M/X)\).
		      \end{enumerate}
	\end{enumerate}
\end{theorem}

Let \(M\) be a basic support \(\tau\)-tilting \(\Lambda\)-module.
We use the notation \(\mathcal{LM}(M)\) to denote the following set:
\begin{equation}
	\left\{X\in \itaurigid \Lambda \setmid \parbox{5cm}{\(X\) is a direct summand of \(M\) and \(X\notin \Fac(M/X)\)}\right\}.
\end{equation}
We remark that \(M/R(M,M)\cong \bigoplus_{X\in \mathcal{LM}(M)} X/R(M,X)\).
\begin{theorem}
	[{\cite[Theorem 2.3, Proposition 2.13]{MR4139031}}]\label{labeling method tau 2021-10-13 17:35:07}
	Let \(M\) be a basic support \(\tau\)-tilting \(\Lambda\)-module.
	Then the following hold:
	\begin{enumerate}
		\item The following map gives an injection:
		      \begin{equation}\label{2021-09-19 23:16:42}
			      \begin{tikzcd}[row sep=1pt]
				      \mathcal{LM}(M)\ar[r]& \brick \Lambda\\
				      X \ar[r,mapsto]& X/R(M,X).
			      \end{tikzcd}
		      \end{equation}
		\item The following map is well-defined and bijective:
		      \begin{equation}\label{2021-09-24 10:32:59}
			      \begin{tikzcd}[row sep=1pt]
				      \mathcal{LM}(M)\ar[r]& \left\{ M'\in \stautilt \Lambda \setmid \text{\(M'\) is a left mutation of \(M\)}\right\}\\
				      X \ar[r,mapsto] & \mu_X(M),
			      \end{tikzcd}
		      \end{equation}
		      here \(\mu_X(M)\) is a unique support \(\tau\)-tilting module having \(M/X\) as a direct summand and being not \(M\).
		      \label{asai brick}
	\end{enumerate}
\end{theorem}
\begin{corollary}
	By \cref{labeling method tau 2021-10-13 17:35:07}, we get the injective map
	\begin{equation}\label{labeling brick 2021-10-13 08:59:12}
		\begin{tikzcd}[row sep=1pt]
			\left\{ M'\in \stautilt \Lambda \setmid \text{\(M'\) is a left mutation of \(M\)}\right\}\ar[r]&\brick \Lambda\\
		\end{tikzcd}
	\end{equation}
	which makes the following diagram commutative:
	\begin{equation}
		\begin{tikzcd}
			\left\{ M'\in \stautilt \Lambda \setmid \text{\(M'\) is a left mutation of \(M\)} \right\}\ar[r,"\eqref{labeling brick 2021-10-13 08:59:12}"]&\brick \Lambda\\
			\mathcal{LM}(M)\ar[u,"\eqref{2021-09-24 10:32:59}"] \ar[ur,"\eqref{2021-09-19 23:16:42}"'].&
		\end{tikzcd}
	\end{equation}
\end{corollary}

\begin{theorem}
	[{\cite[Theorem 2.3]{MR4139031}, \cite[Theorem 4.1, Lemma 4.3]{MR3910476}}]\label{asai correspondence}
	Let \(M\) be a basic support \(\tau\)-tilting \(\Lambda\)-module. Then the following hold:
	\begin{enumerate}
		\item The module \(M/R(M,M)\) is a left finite semibrick.
		\item The following map gives a bijection:
		      \begin{equation}\label{asaicorresp}
			      \begin{tikzcd}[row sep=1pt]
				      \stautilt \Lambda \ar[r]& \flsbrick \Lambda\\
				      M \ar[r,mapsto]& M/R(M,M).
			      \end{tikzcd}
		      \end{equation}
		\item The following map gives a bijection:
		      \begin{equation}\label{TT}
			      \begin{tikzcd}[row sep=1pt]
				      \flsbrick \Lambda \ar[r]& \ftors \Lambda\\
				      S \ar[r,mapsto]& \torscl(S).
			      \end{tikzcd}
		      \end{equation}
		\item The following map gives a bijection:
		      \begin{equation}\label{2021-09-22 14:16:57}
			      \begin{tikzcd}[row sep=1pt]
				      \itaurigid \Lambda \ar[r] &\flbrick \Lambda\\
				      X \ar[r,mapsto]& X/R(X,X).
			      \end{tikzcd}
		      \end{equation}
		\item The following map gives an injection:
		      \begin{equation}\label{2021-09-22 14:17:07}
			      \begin{tikzcd}[row sep=1pt]
				      \flbrick \Lambda \ar[r] &\ftors \Lambda\\
				      S \ar[r,mapsto]& \torscl (S).
			      \end{tikzcd}
		      \end{equation}
		\item The above maps make the following diagram commutative:
		      \begin{equation}
			      \begin{tikzcd}
				      \itaurigid \Lambda \ar[rrr,"\eqref{2021-09-22 14:16:57}"]\ar[dr,"\eqref{2021-09-22 15:21:57}"]\ar[ddr,"\eqref{2021-09-22 14:16:57}"']& & &\flbrick \Lambda \ar[dl,phantom,"\supset"sloped] \ar[ddl,"\eqref{2021-09-22 14:17:07}"]\\
				      &\stautilt \Lambda \ar[d,"\eqref{fac}"] \ar[r,"\eqref{asaicorresp}"]& \flsbrick \Lambda \ar[d,"\eqref{TT}"']&
				      \\
				      &\ftors \Lambda \ar[r,equal]& \ftors \Lambda.&
			      \end{tikzcd}
		      \end{equation}
	\end{enumerate}
\end{theorem}

The dual assertions are also true.

\begin{theorem}
	[{The dual assertion of \cref{2021-11-22 19:52:15}}]\label{labeling method taui 2021-10-13 17:35:43}
	Let \(N\) be a basic support \(\tau^{-1}\)-tilting \(\Lambda\)-module and \(Y\) an indecomposable direct summand of \(N\).
	Then the following hold:
	\begin{enumerate}
		\item The module \(S(Y,Y)\) is a right finite brick over \(\Lambda\).
		\item The module \(S(Y,N)\) is a brick or zero module.
		\item The following conditions are equivalent:
		      \begin{enumerate}
			      \item The module \(S(Y,N)\) is nonzero.
			      \item The module \(S(Y,N)\) is a brick.
			      \item \(Y\notin \Sub(N/Y)\).
		      \end{enumerate}
	\end{enumerate}
\end{theorem}

Let \(N\) be a basic support \(\tau^{-1}\)-tilting module.
We use the notation \(\mathcal{RM}(N)\) to denote the following set:
\begin{equation}
	\left\{Y\in \itauirigid \Lambda\setmid \parbox{5cm}{\(Y\) is a direct summand of \(N\) and \(Y\notin \Sub(N/Y)\)}\right\}.
\end{equation}
We remark that \(S(N,N)\cong \bigoplus_{X\in \mathcal{RM}(N)} S(Y,N)\).

\begin{theorem}
	[{The dual assertion of \cref{labeling method tau 2021-10-13 17:35:07}}]\label{right mutation 2021-11-22 20:43:58}
	Let \(N\) be a basic support \(\tau^{-1}\)-tilting \(\Lambda\)-module. Then the following hold:
	\begin{enumerate}
		\item The following map gives a bijection:
		      \begin{equation}\label{2021-09-24 13:11:04}
			      \begin{tikzcd}[row sep=1pt]
				      \mathcal{RM}(N)\ar[r]& \brick \Lambda\\
				      Y \ar[r,mapsto]& S(Y,N).
			      \end{tikzcd}
		      \end{equation}
		\item The following map gives a bijection:
		      \begin{equation}\label{2021-09-24 15:40:15}
			      \begin{tikzcd}[row sep=1pt]
				      \mathcal{RM}(N)\ar[r]& \left\{ N'\in \stauitilt \Lambda \setmid \text{\(N'\) is a right mutation of \(N\)} \right\}\\
				      Y \ar[r,mapsto]& \mu_Y(N),
			      \end{tikzcd}
		      \end{equation}
		      here \(\mu_Y(N)\) is a unique support \(\tau^{-1}\)-tilting \(\Lambda\)-module having \(N/Y\) as a direct summand and being not \(N\). \label{asai d brick}
	\end{enumerate}
\end{theorem}
\begin{corollary}
	By \cref{right mutation 2021-11-22 20:43:58}, we get the injective map
	\begin{equation}\label{labeling brick 2021-10-13 09:11:22}
		\begin{tikzcd}[row sep=1pt]
			\left\{ N'\in \stauitilt \Lambda \setmid \text{\(N'\) is a right mutation of \(N\)}\right\}\ar[r]&\brick \Lambda\\
		\end{tikzcd}
	\end{equation}
	which make the following diagram commutative:
	\begin{equation}
		\begin{tikzcd}[row sep=large]
			\left\{ N'\in \stauitilt \Lambda \setmid \text{\(N'\) is a right mutation of \(N\)} \right\}\ar[r,"\eqref{labeling brick 2021-10-13 09:11:22}"]&\brick \Lambda\\
			\mathcal{RM}(N)\ar[u,"\eqref{2021-09-24 15:40:15}"] \ar[ur,"\eqref{2021-09-24 13:11:04}"'].&
		\end{tikzcd}
	\end{equation}
\end{corollary}
\begin{theorem}
	[{Dual assertion of \cref{asai correspondence}}]\label{dasai correspondence 2021-09-24 13:18:33}
	Let \(N\) be a basic support \(\tau^{-1}\) tilting \(\Lambda\)-module. Then the following hold:
	\begin{enumerate}
		\item The module \(S(N,N)\) is a right finite semibrick.
		\item The following map gives a bijection:
		      \begin{equation}\label{dasaicorresp 2021-09-24 13:16:37}
			      \begin{tikzcd}[row sep=1pt]
				      \stauitilt \Lambda \ar[r]& \frsbrick \Lambda\\
				      N \ar[r,mapsto]& S(N,N).
			      \end{tikzcd}
		      \end{equation}
		\item The following map gives a bijection:
		      \begin{equation}\label{FF 2021-09-24 13:16:03}
			      \begin{tikzcd}[row sep=1pt]
				      \frsbrick \Lambda \ar[r]& \ftorf \Lambda\\
				      S \ar[r,mapsto]& \torfcl(S).
			      \end{tikzcd}
		      \end{equation}
		\item The following map gives a bijection:
		      \begin{equation}\label{2021-09-24 13:13:50}
			      \begin{tikzcd}[row sep=1pt]
				      \itauirigid \Lambda \ar[r] &\frbrick \Lambda\\
				      Y \ar[r,mapsto]& S(Y,Y).
			      \end{tikzcd}
		      \end{equation}
		\item The following map gives an injection:
		      \begin{equation}\label{2021-09-24 13:14:13}
			      \begin{tikzcd}[row sep=1pt]
				      \frbrick \Lambda \ar[r] &\ftorf \Lambda\\
				      S \ar[r,mapsto]& \torfcl (S).
			      \end{tikzcd}
		      \end{equation}
		\item The above maps make the following diagram commutative:
		      \begin{equation}
			      \begin{tikzcd}
				      \itauirigid \Lambda \ar[rrr,"\eqref{2021-09-24 13:13:50}"]\ar[dr,"\eqref{2021-09-24 13:15:11}"]\ar[ddr,"\eqref{2021-09-24 13:13:50}"']& & &\frbrick \Lambda \ar[dl,phantom,"\supset"sloped] \ar[ddl,"\eqref{2021-09-24 13:14:13}"]\\
				      &\stauitilt \Lambda \ar[d,"\eqref{sub}"] \ar[r,"\eqref{dasaicorresp 2021-09-24 13:16:37}"]& \frsbrick \Lambda \ar[d,"\eqref{FF 2021-09-24 13:16:03}"']&
				      \\
				      &\ftorf \Lambda \ar[r,equal]& \ftorf \Lambda.&
			      \end{tikzcd}
		      \end{equation}
	\end{enumerate}
\end{theorem}

\begin{corollary}\label{comm dbrick 2021-10-01 20:08:26}
	By \cref{dugger 2021-09-06 13:33:14,asai correspondence,dasai correspondence 2021-09-24 13:18:33,indec tau taui corr2021-09-09 16:44:38}, we get the bijective maps
	\begin{equation}\label{dualsbrick}
		\begin{tikzcd}
			\flsbrick \Lambda\ar[r]&\frsbrick \Lambda,
		\end{tikzcd}
	\end{equation}
	\begin{equation}\label{dualbrick}
		\begin{tikzcd}
			\flbrick \Lambda\ar[r]&\frbrick \Lambda,
		\end{tikzcd}
	\end{equation}
	which make the following diagram commutative:
	\begin{equation}\label{cdtaubricktor 2021-09-06 15:51:24}
		\begin{tikzcd}[row sep=20pt]
			&[-10pt]&[-30pt] \flbrick \Lambda \ar[dl,phantom,"\supset"sloped] \ar[dll,near start,"\eqref{2021-09-22 14:17:07}"'] \ar[rr,dashrightarrow,"\eqref{dualbrick}"] \ar[dd,leftarrow,"\eqref{2021-09-22 14:16:57}",very near end,xshift=-10pt] &[-35pt] &[-35pt] \frbrick \Lambda \ar[dl,phantom,"\supset"sloped] \ar[dd, leftarrow,"\eqref{2021-09-24 13:13:50}",near end] \ar[r,"\eqref{2021-09-24 13:13:50}"]&[-10pt]\ftorf \Lambda \ar[dd, equal]\\
			\ftors \Lambda \ar[dd,equal]\ar[urrrrr,bend left=35pt,"\eqref{gal}"]&\ar[l,"\eqref{TT}"]\flsbrick \Lambda \ar[rr,dashrightarrow, crossing over,"\eqref{dualsbrick}",near end] & & \frsbrick \Lambda \ar[urr,crossing over,"\eqref{FF 2021-09-24 13:16:03}"'] &\\
			&& \itaurigid \Lambda \ar[dl,"\eqref{2021-09-22 15:21:57}"description]\ar[dll,"\eqref{2021-09-22 14:16:57}"'] \ar[rr,near start,"\eqref{tau or top 2021-09-11 19:19:09}"] & & \itauirigid \Lambda \ar[dl,pos=0.4,"\eqref{2021-09-24 13:13:50}"description] \ar[r,"\eqref{2021-09-24 13:13:18}"]&\ftorf \Lambda \\
			\ftors \Lambda&\stautilt \Lambda \ar[l,"\eqref{fac}"] \ar[uu,pos=0.7,crossing over,"\eqref{asaicorresp}"]\ar[rr,"\eqref{ddual}"] & & \stauitilt \Lambda \ar[uu,rightarrow,very near start,xshift=9pt,crossing over,"\eqref{dasaicorresp 2021-09-24 13:16:37}"] \ar[urr,"\eqref{sub}"'].&\\
		\end{tikzcd}
	\end{equation}

\end{corollary}
\begin{definition}
	[{\cite[Definition 2.29]{MR3187626}, \cite[Definition 2.14]{MR4139031}}]\label{Hasse rabeled by brick}
	We define the support \(\tau\)-tilting Hasse quiver \(\mathcal{H}(\stautilt \Lambda)\) labeled with brick over \(\Lambda\) as follows:
	\begin{itemize}
		\item The set of vertices is \(\stautilt \Lambda\).
		\item We draw an arrow from \(M\) to each support \(\tau\)-tilting left mutation \(M'\) of \(M\), and we label this arrow with the corresponding brick to \(M'\) under \eqref{labeling brick 2021-10-13 08:59:12} for \(M\).
	\end{itemize}
\end{definition}

\begin{definition}
	\label{Hasse i rabeled by brick}
	We define the support \(\tau^{-1}\)-tilting Hasse quiver \(\mathcal{H}(\stauitilt \Lambda)\) labeled with bricks over \(\Lambda\) as follows:
	\begin{itemize}
		\item The set of vertices is \(\stauitilt \Lambda\).
		\item We draw an arrow from \(N\) to each support \(\tau\)-tilting right mutation \(N'\) of \(N\), and we label this arrow with the corresponding brick to \(N'\) under \eqref{labeling brick 2021-10-13 09:11:22} for \(N\).
	\end{itemize}
\end{definition}

\begin{proposition}[{\cite[Lemma 2.16, Proposition 2.17]{MR4139031}}]\label{labeling properties 2021-11-14 17:10:59}
	Let \(M\) be a support \(\tau\)-tilting module and \(M'\) be a support \(\tau\)-tilting left mutation of \(M\).
	Furthermore, \(N\) and \(N'\) be support \(\tau^{-1}\)-tilting module corresponding to \(M\) and \(M'\) under \eqref{ddual}, respectively.
	Then there exist a unique brick in the subcategory \(\Fac M \cap \Sub N'\), and it is isomorphic to the brick corresponding to \(M'\) under \eqref{2021-09-19 23:16:42} for \(M\).
\end{proposition}
\begin{theorem}
	[{\cite[Theorem 2.15, Lemma 2.16, Proposition 2.17]{MR4139031}}]
	Let \(M\) be a support \(\tau\)-tilting \(\Lambda\)-module, \(M'\) a support \(\tau\)-tilting left mutation of \(M\), \(S\) a brick on the arrow \(M\rightarrow M'\) in \(\mathcal{H}(\stautilt \Lambda)\).
	Moreover, \(N\) and \(N'\) be support \(\tau^{-1}\)-tilting \(\Lambda\)-modules corresponding to \(M\) and \(M'\) under \eqref{ddual}, respectively.
	Then \(N\) is a right mutation of \(N'\) and \(S\) is also the brick labeled with the arrow \(N'\rightarrow N\) in \(\mathcal{H}(\stauitilt \Lambda)\).
\end{theorem}

\subsection{Silting complexes and simple-minded collections}
We recall the definition of silting complexes, which is a generalization of tilting complexes. The concept of silting complexes is originated from \cite{keller1988aisles}, and recently there have been many papers on silting complexes and silting mutations starting with \cite{MR2927802}.
In particular, in \cite{MR3187626}, it is shown that there is a one-to-one correspondence between the two-term silting complexes and the support \(\tau\)-tilting modules.

\begin{definition}
	Let \(T\) be a complex in \(K^b(\Lambda\lproj)\).
	\begin{enumerate}
		\item We say that \(T\) is presilting (or pretilting) if \[\Hom_{K^b(\Lambda\lproj)}(T,T[i])=0\] for any \(i>0\) (or for any \(i \neq 0\), respectively).
		\item We say that \(T\) is silting (or  tilting) if it is presilting (or pretilting, respectively) and satisfies \(\thick T =K^b(\Lambda\lproj)\), where \(\thick T\) means the smallest triangulated subcategory of \(K^b(\Lambda\lproj)\) containing \(\add T\).
	\end{enumerate}
\end{definition}

We say that a complex \(T \in K^b(\Lambda\lproj)\) is two-term if \(T^i = 0\) for all \(i \neq 0, -1\).
Moreover, we use the following notation:
\begin{itemize}
	\item \(\silt \Lambda\) means the set of isomorphism classes of basic silting complexes in \(K^b(\Lambda\lproj)\),
	\item \(\tilt \Lambda\) means the set of isomorphism classes of basic tilting complexes in \(K^b(\Lambda\lproj)\),
	\item \(\twosilt \Lambda\) means the set of isomorphism classes of basic two-term silting complexes in \(K^b(\Lambda\lproj)\),
	\item \(\twotilt \Lambda\) means the set of isomorphism classes of basic two-term tilting complexes in \(K^b(\Lambda\lproj)\).
\end{itemize}

We can define a partially ordered set structure on \(\silt \Lambda\) as follows.

\begin{definition}[{\cite[Definition 2.10, Theorem 2.11]{MR2927802}}]
	For \(T, T'\in \silt \Lambda\), we write \(T\geq T'\) if
	\begin{equation}
		\Hom_{K^b(\Lambda\lproj)}(T,T'[i])=0
	\end{equation}
	for any \(i>0\).
\end{definition}

\begin{theorem}[{\cite[Theorem 3.2 and Corollary 3.9]{MR3187626}}]\label{2020-03-27 17:22:17}
	The following map gives an isomorphism of partially ordered sets:
	\begin{equation}\label{2siltcorr}
		\begin{tikzcd}[row sep=1pt]
			\stautilt \Lambda \ar[r] & \twosilt \Lambda\\
			M\ar[r,mapsto]& (P_1\oplus P \xrightarrow{(f_1\ 0)} P_0),
		\end{tikzcd}
	\end{equation}
	here \(P_1\xrightarrow{f_1}P_0\xrightarrow{f_0}M\rightarrow 0\)
	is a minimal projective presentation of \(M\) and \(P\) is a basic projective module satisfying that \(\Hom_\Lambda(P,M)=0\) and that \(|M|+|P|=|\Lambda|\) (see \cite[Proposition 2.3 (b)]{MR3187626}).
\end{theorem}

We remark that the correspondence above commutes with support \(\tau\)-tilting mutations and silting mutations \cite[Corollary 3.9]{MR3187626}.

\begin{proposition}[{\cite[Example 2.8]{MR2927802}}]
	If \(\Lambda\) is a finite dimensional symmetric \(k\)-algebra, then any silting complex in \(K^b(\Lambda \lproj)\) is a tilting complex.
\end{proposition}
\begin{definition}
	We say that
	a set \(\mathcal{X}\) of isomorphism classes of objects in \(D^b(\Lambda\lmod)\) is a simple-minded collection in \(D^b(\Lambda\lmod)\) if it satisfies the following conditions:
	\begin{enumerate}
		\item For any \(X \in \mathcal{X}\), we have  \[\End_{D^b(\Lambda\lmod)}(X)\cong k.\]
		\item For any \(X_1, X_2 \in \mathcal{X}\) with \(X_1\neq X_2\), we have \[\Hom_{D^b(\Lambda\lmod)}(X_1, X_2)=0.\]
		\item For any \(X_1, X_2 \in \mathcal{X}\) and \(t <0\), we have \[\Hom_{D^b(\Lambda\lmod)}(X_1,X_2[t])=0.\]
		\item The triangulated subcategory \(\thick \mathcal{X}\) coincides with \(D^b(\Lambda\lmod)\).
	\end{enumerate}
\end{definition}
We say that a simple-minded collection \(\mathcal{X}\) in \(D^b(\Lambda\lmod)\) is two-term if \(\mathcal{X}\) satisfies the condition that \(H^i(X)=0\) for any \(i \neq -1,0\) and any \(X \in \mathcal{X}\).
Moreover, we use the following notation:
\begin{itemize}
	\item \(\smc \Lambda\) means the set of simple-minded collections in \(D^b(\Lambda\lmod)\),
	\item \(\twosmc \Lambda\) means the set of two-term simple-minded collections in \(D^b(\Lambda\lmod)\).
\end{itemize}
\begin{proposition}[{\cite[Remark 4.11]{MR3220536},
			}]
	For any two-term simple-minded collection \(\mathcal{X}\) in \(D^b(\Lambda\lmod)\),
	every \(X \in \mathcal{X}\) belongs to either \(\Lambda\lmod\) or \((\Lambda\lmod)[1]\) up to isomorphisms in \(D^b(\Lambda\lmod)\).
\end{proposition}

\begin{proposition}[{\cite[Theorem 3.3]{MR4139031}}]\label{asaibj}
	The following maps give bijections:
	\begin{equation}\label{asaibjfl}
		\begin{tikzcd}[row sep=1pt]
			\twosmc \Lambda \ar[r]&\flsbrick \Lambda\\
			\mathcal{X}\ar[r,mapsto]&\bigoplus_{i}S_i,
		\end{tikzcd}
	\end{equation}
	\begin{equation}\label{asabjfr}
		\begin{tikzcd}[row sep=1pt]
			\twosmc \Lambda \ar[r]&\frsbrick \Lambda\\
			\mathcal{X}\ar[r,mapsto]&\bigoplus_{i}R_i,
		\end{tikzcd}
	\end{equation}
	where \(S_i\) are objects in \(\mathcal{X}\cap \Lambda\lmod\) and \(R_i\) are objects in \(\mathcal{X}[-1]\cap \Lambda\lmod\).
\end{proposition}

Now we recall the construction of silting complexes from simple-minded collections based on \cite{MR3178243,MR1947972}.
Let \(\mathcal{X}=\left\{X_1, \ldots, X_r \right\}\) be a simple-minded collection in \(D^b(\Lambda\lmod)\). By induction on \(n\), we shall construct sequences
\begin{equation}
	\begin{tikzcd}[column sep=0.6cm]
		X_i^{(0)} \ar[r,"\beta_i^{(0)}"]& X_i^{(1)} \ar[r,"\beta_i^{(1)}"]& X_i^{(2)} \ar[r]& \cdots \ar[r]&X_i^{(n)} \ar[r,"\beta_i^{(n)}"]&X_i^{(n+1)}\ar[r]&\cdots
	\end{tikzcd}
\end{equation}
of objects and morphisms in \(D^b(\Lambda\lMod)\) for \(1\leq i\leq r\). Set \(X_i^{(0)}:=X_i\). Suppose we have constructed \(X_i^{(n-1)}\). For each \(1\leq j\leq r\) and \(t<0\), choose a basis \(B^{(n)}(j,t,i)\) of \(\Hom_{D^b(\Lambda\lmod)}(X_j[t],X_i^{(n-1)})\). Put
\begin{equation}
	Z_i^{(n-1)}=\bigoplus_{t<0}\bigoplus_{j=1}^{r}\bigoplus_{f\in B^{(n)}(j,t,i)}X_j[t]
\end{equation}
and let \(\alpha_i^{(n-1)}\colon Z_i^{(n-1)} \rightarrow X_i^{(n-1)}\) be the map whose restriction to the component indexed by \(t<0\), \(j=1, \ldots, r\), \(f \in B^{(n)}(j,t,i)\) is exactly \(f\).
Now define \(X_i^{(n)}\) together with morphism
\(
\begin{tikzcd}
	X_i^{(n-1)}\ar[r,"\beta_i^{(n-1)}"]&X_i^{(n)}
\end{tikzcd}
\)
by forming the distinguished triangle
\begin{equation}
	\begin{tikzcd}
		Z_i^{(n-1)}\ar[r,"\alpha_i^{(n-1)}"]&X_i^{(n-1)}\ar[r,"\beta_i^{(n-1)}"]&X_i^{(n)}\ar[r]&Z_i^{(n)}[1].
	\end{tikzcd}
\end{equation}
Let \(C_i\) be the homotopy colimit of the sequence
\begin{equation}
	\begin{tikzcd}[column sep=0.7cm]
		X_i^{(0)} \ar[r,"\beta_i^{(0)}"]& X_i^{(1)} \ar[r,"\beta_i^{(1)}"]& X_i^{(2)} \ar[r]& \cdots \ar[r]&X_i^{(n-1)} \ar[r,"\beta_i^{(n-1)}"]&X_i^{(n)}\ar[r]&\cdots
	\end{tikzcd}
\end{equation}
and put \(T^{\mathcal{X}}=\nu^{-1}(\bigoplus_i C_i)\).
This construction induces a well-defined bijection between \(\smc \Lambda\) and \(\silt \Lambda\).
\begin{theorem}[{\cite[Proposition 5.9, Theorem 6.1]{MR3178243}, \cite[Corollary 4.3]{MR3220536}}]
	The following map gives a bijection:
	\begin{equation}\label{KY}
		\begin{tikzcd}[row sep=1pt]
			\smc \Lambda\ar[r]& \silt \Lambda\\
			\mathcal{X}\ar[r,mapsto]&T^\mathcal{X}.
		\end{tikzcd}
	\end{equation}
	This bijection restricts to a bijection:
	\begin{equation}\label{2KYcorr}
		\twosmc \Lambda \rightarrow \twosilt \Lambda.
	\end{equation}
\end{theorem}
As can be seen from the definition, it is difficult in general to calculate \(T^{\mathcal{X}}\) from \(\mathcal{X}\).
However, if we restrict the discussions to two-term simple-minded collections and two-term silting complexes, the following theorem make the above computation easier via the correspondence between support \(\tau\)-tilting modules and left finite semibricks.
\begin{theorem}[{\cite[Theorem 3.3]{MR4139031}}]\label{commutative}
	The following diagram is commutative:
	\begin{equation}
		\begin{tikzcd}
			\twosmc \Lambda \ar[r,"\eqref{2KYcorr}"] \ar[d,"\eqref{asaibjfl}"'] &\twosilt \Lambda \ar[d,leftarrow,"\eqref{2siltcorr}"]\\
			\flsbrick \Lambda \ar[r,leftarrow,"\eqref{asaicorresp}"'] & \stautilt \Lambda.
		\end{tikzcd}
	\end{equation}
\end{theorem}

The following theorem enables us to calculate the corresponding two-term simple-minded collections from left finite semibricks (or right finite semibricks).

\begin{theorem}	[{\cite[Theorem 3.3]{MR4139031}}]\label{sbricksmc 2021-09-06 15:11:30}
	The following diagram is commutative:
	\begin{equation}
		\begin{tikzcd}
			\twosmc \Lambda \ar[r,"\eqref{asabjfr}"] \ar[d,"\eqref{asaibjfl}"'] &\frsbrick \Lambda \ar[dl,"\eqref{dualsbrick}"]\\
			\flsbrick \Lambda. &
		\end{tikzcd}
	\end{equation}
\end{theorem}

\section{Preliminaries of modular representation theory of finite groups}
In this section, let \(k\) be an algebraically closed field of characteristic \(p>0\).
For any finite group \(G\), the field \(k\) can always be regarded as a \(kG\)-module by defining \(gx=x\) for
any \(g\in G\) and \(x \in k\).
This module is called the trivial module and is denoted by \(k_G\).
For \(kG\)-modules \(U\) and \(V\), the \(k\)-module \(U\otimes V =U \otimes_k V\) has a \(kG\)-module structure given by \(g(u\otimes v)=gu\otimes gv\) for all \(g \in G\), \(u \in U\) and \(v \in V\).
\subsection{Restriction functors and induction functors}
Let \(G\) be a finite group and \(H\) a subgroup of \(G\).
We denote by \(\restr_{H}^G\) the restriction functor from \(kG\lmod\) to \(kH\lmod\) and \(\induc_H^G:={}_{kG}kG\otimes _{kH} \bullet\) the induction functor from \(kH\lmod\) to \(kG\lmod\).
The functors \(\restr_{H}^G\) and \(\induc_H^G\) are exact functors and have the following properties.
\begin{proposition}[{see \cite[Lemma 8.5, Lemma 8.6]{MR860771}}]\label{Theorem: Frobenius and projective}
	Let \(G\) be a finite group, \(K\) a subgroup of \(G\), \(H\) a subgroup of \(K\), \(U\) a \(kG\)-module and \(V\) a \(kH\)-module.
	Then the following hold:
	\begin{enumerate}
		\item \(\restr^G_K\restr^K_H \cong\restr^G_H\).
		\item \(\induc^G_K\induc^K_H \cong\induc^G_H\).
		\item The functors \(\restr_H^G\) and \(\induc_H^G\) are left and right adjoint to each other.
		\item The functors \(\restr_H^G\) and \(\induc_H^G\) send projective modules to projective modules.
		\item \(U\otimes k_G\cong U\).\label{triv tensor 2021-12-06 10:47:29}
		\item \(\induc_H^G(\restr_H^G U\otimes V)\cong U\otimes \induc_H^G V\).\label{frob tensor 2021-12-06 10:50:14}
	\end{enumerate}
\end{proposition}

Let \(\tilde{G}\) be a finite group, \(G\) a normal subgroup of \(\tilde{G}\) and \(U\) a \(kG\)-module. For \(\tilde{g} \in \tilde{G}\), we define a \(kG\)-module \(\tilde{g}U\) consisting of symbols \(\tilde{g}u\) as a set, where \(u\in U\) and its \(kG\)-module structure is given by \(\tilde{g}u+\tilde{g}u':=\tilde{g}(u+u')\), \(\lambda(\tilde{g}u):=\tilde{g}(\lambda u)\) and \(g(\tilde{g}u):=\tilde{g}(\tilde{g}^{-1}g\tilde{g}u) \) for any \(u, u'\in U\), \(\lambda\in k\) and \(g \in G\).

\begin{theorem}[Mackey's decomposition formula for normal subgroups]\label{Mackey's decomposition formula}
	Let \(\tilde{G}\) be a finite group and \(G\) a normal subgroup of \(G\), and \(U\) a \(kG\)-module. Then we have
	\begin{equation}
		\restr^{\tilde{G}}_G\induc^{\tilde{G}}_G U\cong\bigoplus_{t \in [\tilde{G}/G]} tU,
	\end{equation}
	where \([\tilde{G}/G]\) is a set of representatives of the factor group \(\tilde{G}/G\).
\end{theorem}
For a \(kG\)-module \(U\), we denote by \(\decompgp_{\tilde{G}}(U)\) the inertial group of \(U\) in \(\tilde{G}\), that is
\begin{equation}
	\decompgp_{\tilde{G}}(U):=\left\{ \tilde{g} \in \tilde{G} \setmid \tilde{g}U \cong U \text{ as \(kG\)-modules} \right\}.
\end{equation}

\begin{theorem}[Clifford's Theorem for simple modules]\label{Clifford's Theorem of simple module}
	Let \(\tilde{G}\) be a finite group, \(G\) a normal subgroup of \(\tilde{G}\), \(S\) a simple \(k\tilde{G}\)-module and \(S'\) a simple \(kG\)-submodule of \(\restr_N^G S\). Then we have a \(kG\)-module isomorphism
	\begin{equation}
		\restr_G^{\tilde{G}} S \cong \bigoplus_{\tilde{g} \in [\tilde{G}/\decompgp_{\tilde{G}}(S')]} \tilde{g}S'^{\oplus r}
	\end{equation}
	for some integer \(r\), which is called the ramification index of \(S\) in \(\tilde{G}\).
\end{theorem}

\subsection{Indecomposable modules, bricks and simple modules satisfying the stable conditions}
Let \(\tilde{G}\) be a finite group, \(G\) a normal subgroup of \(\tilde{G}\) and \(U\) a \(kG\)-module.
If \(\decompgp_{\tilde{G}}(U)={\tilde{G}}\), we say that \(U\) is \({\tilde{G}}\)-stable, that is, for any \(\tilde{g}\in \tilde{G}\) the \(kG\)-module \(\tilde{g}U\) is isomorphic to \(U\).
For a \(k\tilde{G}\)-module \(\tilde{U}\), we say that \(\tilde{U}\) is an extending \(k\tilde{G}\)-module of \(U\) if \(\restr_G^{\tilde{G}} \tilde{U}\cong U\).
\begin{remark}\label{rem extending action hom}
	We remark that any extending \(k\tilde{G}\)-module of indecomposable \(kG\)-modules (respectively, of bricks in \(kG\lmod\), of simple \(kG\)-modules) is also an indecomposable \(kG\)-module (respectively, a brick in \(k\tilde{G}\lmod\) and a simple \(k\tilde{G}\)-module).
	Let \(U_i\) be a \(kG\)-module and \(\tilde{U_i}\) an extending \(k\tilde{G}\)-module of \(U_i\) (\(i=1, 2\)).
	We also remark that if \(U_1\ncong U_2\) then \(\tilde{U_1}\ncong \tilde{U_2}\) and if \(\Hom_{kG}(U_1,U_2)=0\) then \(\Hom_{k\tilde{G}}(\tilde{U}_1,\tilde{U}_2)=0\).
\end{remark}

\begin{lemma}
	\label{rem ind and extending module 2021-09-29 08:59:16}
	Let \(\tilde{G}\) be a finite group, \(G\) a normal subgroup, \(U\) a \(kG\)-module and \(\tilde{U}\) an extending \(k\tilde{G}\)-module of \(U\).
	Then the following hold:
	\begin{enumerate}
		\item \(\induc_G^{\tilde{G}}U\cong \tilde{U}\otimes k[\tilde{G}/G].\)\label{frob b 2021-10-25 13:13:50}\label{2021-12-06 12:14:15}
		\item \(\torscl(\induc_G^{\tilde{G}}U)=\torscl(\bigoplus_V(\tilde{U}\otimes V))\), where \(V\) runs through representatives of the isomorphism classes of simple \(k[\tilde{G}/G]\)-modules.\label{2021-12-06 12:15:21}
		\item \(\torfcl(\induc_G^{\tilde{G}}U)=\torfcl(\bigoplus_V(\tilde{U}\otimes V))\), where \(V\) runs through representatives of the isomorphism classes of simple \(k[\tilde{G}/G]\)-modules.\label{2021-12-06 12:16:08}
	\end{enumerate}
\end{lemma}
\begin{proof}
	The assertion \ref{2021-12-06 12:14:15} follows from the following isomorphisms:
	\begin{equation}
		\induc_G^{\tilde{G}} U\cong \induc_G^{\tilde{G}}( \restr_G^{\tilde{G}}\tilde{U}\otimes k_G)\cong  \tilde{U}\otimes \induc_G^{\tilde{G}}k_{\tilde{G}}
		\cong\tilde{U}\otimes k[\tilde{G}/G],
	\end{equation}
	here the first and second isomorphisms come from \cref{Theorem: Frobenius and projective}.

	We prove the assertion \ref{2021-12-06 12:15:21}.
	By taking a composition series
	\begin{equation}
		0=W_0\subset W_1\subset \cdots \subset W_{e-1}\subset W_e={}_{k\tilde{G}}k[\tilde{G}/G]
	\end{equation}
	of the module \({}_{k\tilde{G}}k[\tilde{G}/G]\), we have the filtration
	\begin{equation}
		0=\tilde{U}\otimes W_0 \subset \tilde{U}\otimes W_1 \subset \cdots \subset \tilde{U}\otimes W_{e-1} \subset \tilde{U}\otimes W_e =\tilde{U}\otimes {}_{k\tilde{G}}k[\tilde{G}/G]
	\end{equation}
	of \(\tilde{U}\otimes k[\tilde{G}/G]\cong \induc_G^{\tilde{G}}U\).
	Therefore, we get \(\torscl(\induc_G^{\tilde{G}} U)= \torscl(\bigoplus_V (\tilde{U}\otimes V))\) by \cref{filt and closure 2021-12-06 11:12:23}.

	The dual arguments show that the assertion \ref{2021-12-06 12:16:08} is also true.
\end{proof}
The same arguments of \cite[Theorems 3.5.7, Theorem 3.5.8 and Corollary 3.5.9]{MR998775} work for bricks.
Hence, we have the following.
\begin{proposition}[{Brick version of \cite[Corollary 3.5.9]{MR998775}}]\label{extending action}
	Let \(\tilde{G}\) be a finite group, \(G\) a normal subgroup of \(\tilde{G}\) and \(S\) a brick in \(kG\lmod\).
	Moreover, let \(e\) be the number of isomorphism classes of \(1\)-dimensional \(k[\tilde{G}/G]\)-modules.
	Assume the following conditions hold:
	\begin{itemize}
		\item \(S\) is \(\tilde{G}\)-stable,
		\item \(H^2(\tilde{G}/G,k^\times)=1\).
	\end{itemize}
	Then there exist \(e\) isomorphism classes of extending \(k\tilde{G}\)-modules \(\tilde{S}^{(1)}, \ldots, \tilde{S}^{(e)}\) of \(S\).
	Moreover, for \(\tilde{S}^{(i)}\) and \(\tilde{S}^{(j)}\), there exists a unique \(1\)-dimensional \(k[\tilde{G}/G]\)-module \(V\) such that \(\tilde{S}^{(i)}\otimes V\cong \tilde{S}^{(j)}\) up to isomorphisms.
\end{proposition}
\begin{proof}
	Let \(\xi_S\colon G\rightarrow GL_k(S)\) be the linear representation of \(G\) corresponding to \(kG\)-module \(S\) and \(X\) a representative of \(\tilde{G}/G\) containing the unit element \(1_{\tilde{G}}\) of \(\tilde{G}\).
	For any \(1_{\tilde{G}}\neq x\in X\), we can take an isomorphism \(\psi_x\colon xS\rightarrow S\) as \(kG\)-modules, and we define \(\psi_{1_{\tilde{G}}}=\Id_S\).
	For any \(x\in X\), we define the \(k\)-linear map \(\varphi_x \in GL_k(S)\) of \(S\) by \(s\mapsto \varphi_x(s):=\psi_x(xs)\), then we get that \(\xi_S(g)\varphi_x=\varphi_x\xi_S(g^x)\) for any \(x\in X\) and \(g\in G\), here we mean that \(g^x=x^{-1}gx\).

	Since every element \(\tilde{g}\) of \(\tilde{G}\) can be expressed uniquely as \(\tilde{g}=xg\) with \(g \in G\) and \(x\in X\), it is possible to define the map \(\hat{\xi}: \tilde{G} \rightarrow GL_k(S)\) by \(xg\mapsto \varphi_x \xi_S(g)\), which is extending map of \(\xi_S\).
	For any \(g,h\in G\) and \(x\in X\), we have the following:
	\begin{align}
		\xi_S(h)\hat{\xi}(xg)
		 & = \xi_S(h)\varphi_x\xi_S(g)        \\
		 & = \varphi_x\xi_S(h^x)\xi_S(g)      \\
		 & = \varphi_x\xi_S(h^xg)             \\
		 & = \varphi_x\xi_S(x^{-1}hxg)        \\
		 & = \varphi_x\xi_S(gg^{-1}x^{-1}hxg) \\
		 & = \varphi_x\xi_S(gh^{xg})          \\
		 & = \varphi_x\xi_S(g)\xi_S(h^{xg})   \\
		 & = \hat{\xi}(xg)\xi_S(h^{xg}).
	\end{align}
	Therefore, we get
	\begin{equation}
		\xi_S(h)\hat{\xi}(\tilde{g})=\hat{\xi}(\tilde{g})\xi_S(h^{\tilde{g}})
	\end{equation}
	for any \(\tilde{g}\in \tilde{G}\) and \(h\in G\).
	Hence, for any \(\tilde{g}, \tilde{h}\in \tilde{G}\), the \(k\)-linear map
	\(\hat{\xi}(\tilde{g})\hat{\xi}(\tilde{h})\hat{\xi}(\tilde{g}\tilde{h})^{-1}\) of \(S\) is an endomorphism of \(kG\)-module \(S\) since it holds that
	\begin{align}
		\hat{\xi}(\tilde{g})\hat{\xi}(\tilde{h})\hat{\xi}(\tilde{g}\tilde{h})^{-1}\xi_S(g)
		 & =\hat{\xi}(\tilde{g})\hat{\xi}(\tilde{h})\xi_S(g^{\tilde{g}\tilde{h}})\hat{\xi}(\tilde{g}\tilde{h})^{-1} \\
		 & =\hat{\xi}(\tilde{g})\xi_S(g^{\tilde{g}})\hat{\xi}(\tilde{h})\hat{\xi}(\tilde{g}\tilde{h})^{-1}          \\
		 & =\xi_S(g)\hat{\xi}(\tilde{g})\hat{\xi}(\tilde{h})\hat{\xi}(\tilde{g}\tilde{h})^{-1}
	\end{align}
	for any \(g\in G\). Now the \(kG\)-module \(S\) is a brick in \(kG\lmod\), hence we can take a scalar \(\alpha(\tilde{g},\tilde{h})\in k^\times\) such that
	\begin{equation}
		\hat{\xi}(\tilde{g})\hat{\xi}(\tilde{h})\hat{\xi}(\tilde{g}\tilde{h})^{-1}=\alpha(\tilde{g},\tilde{h})\Id_S,
	\end{equation}
	and that \(\alpha(1_{\tilde{G}},1_{\tilde{G}})=1\).
	Moreover, for any \(g,h\in G\) and \(x,y \in X\), we get \(\alpha(gx,hy)=\alpha(x,y)\) as follows:
	\begin{align}
		\alpha(xg,yh)\Id_S
		 & =\hat{\xi}(xg)\hat{\xi}(yh)\hat{\xi}(xgyh)^{-1}                                                \\
		 & =\varphi_x\xi_S(g)\varphi_y\xi_S(h)\hat{\xi}(xyg^yh)^{-1}                                      \\
		 & =\varphi_x\xi_S(g)\varphi_y\xi_S(h)\hat{\xi}(zfg^yh)^{-1}            & (xy=zf, z\in X, f\in G) \\
		 & =\varphi_x\xi_S(g)\varphi_y\xi_S(h)(\varphi_z\xi_S(fg^yh))^{-1}                                \\
		 & =\varphi_x\varphi_y\xi_S(g^y)\xi_S(h)\xi_S(fg^yh)^{-1}\varphi_z^{-1}                           \\
		 & =\varphi_x\varphi_y\xi_S(f)^{-1}\varphi_z^{-1}                                                 \\
		 & =\varphi_x\varphi_y\hat{\xi}(zf)^{-1}                                                          \\
		 & =\hat{\xi}(x)\hat{\xi}(y)\hat{\xi}(xy)^{-1}                                                    \\
		 & =\alpha(x,y)\Id_S.
	\end{align}
	We can check that the map \(\alpha\colon \tilde{G}\times\tilde{G}\rightarrow k^\times\) satisfies the \(2\)-cocycle condition easily (for example, see \cite{MR998775}).
	Hence, \(\alpha\) induces the map \(\bar{\alpha}\colon \tilde{G}/G\times\tilde{G}/G\rightarrow k^\times\) satisfying the \(2\)-cocycle condition.
	Since \(H^2(\tilde{G}/G,k^\times)\) is trivial, there exists a map \(\beta\colon \tilde{G}/G\rightarrow k^\times\) satisfying that for any \(\tilde{g}, \tilde{h}\in \tilde{G}\),
	\begin{equation}
		\beta(\tilde{g}G)\beta(\tilde{g}\tilde{h}G)^{-1}\beta(\tilde{h}G)=\bar{\alpha}(\tilde{g}G,\tilde{h}G)=\alpha(\tilde{g},\tilde{h}).
	\end{equation}
	We define \(\tilde{\xi} \colon \tilde{G} \rightarrow GL_k(S)\) by \(\tilde{\xi}(\tilde{g})=\beta(\tilde{g}G)\hat{\xi}(\tilde{g})\).
	Then \(\tilde{\xi}\) is a group homomorphism and an extension of \(\xi_S\) since \(\beta(1_{\tilde{G}}G)=1\).
	Hence, the corresponding \(k\tilde{G}\)-module \(\tilde{S}\) to \(\tilde{\xi}\) is an extending \(k\tilde{G}\)-module of the \(kG\)-module \(S\).
	Let \(W\) be another extending \(k\tilde{G}\)-module of \(S\) and \(\xi_W\colon \tilde{G}\rightarrow GL_k(W)\) the  linear representation of \(\tilde{G}\) corresponding to \(W\).
	Then there exists a \(kG\)-isomorphism \(\nu\colon \restr_G^{\tilde{G}}W \rightarrow S\), that is, the following equation holds for any \(g\in G\):
	\begin{equation}\label{eq 2021-12-16 19:47:31}
		\nu \xi_W(g)=\xi_S(g)\nu.
	\end{equation}
	For \(\tilde{g}\in \tilde{G}\), we define the following linear map of \(S\):
	\begin{equation}
		\gamma(\tilde{g}):=\nu\xi_W(\tilde{g})\nu^{-1}\tilde{\xi}(\tilde{g})^{-1}.
	\end{equation}
	If \(g\in G\) then \(\gamma(g)\) is the identity map by \eqref{eq 2021-12-16 19:47:31}, and
	we can show that \(\gamma(\tilde{g})\) is an endomorphism of \(kG\)-module \(S\) as follows:
	\begin{align}
		\gamma(\tilde{g})\xi_S(g)
		 & =\nu\xi_W(\tilde{g})\nu^{-1}\tilde{\xi}(\tilde{g})^{-1}\xi_S(g)                                                  \\
		 & =\nu\xi_W(\tilde{g})\nu^{-1}\tilde{\xi}(\tilde{g})^{-1}\xi_S(g)\tilde{\xi}(\tilde{g})\tilde{\xi}(\tilde{g})^{-1} \\
		 & =\nu\xi_W(\tilde{g})\nu^{-1}\tilde{\xi}(\tilde{g}^{-1}g\tilde{g})\tilde{\xi}(\tilde{g})^{-1}                     \\
		 & =\nu\xi_W(\tilde{g})\xi_W(\tilde{g}^{-1}g\tilde{g})\nu^{-1}\tilde{\xi}(\tilde{g})^{-1}                           \\
		 & =\nu\xi_W(g\tilde{g})\nu^{-1}\tilde{\xi}(\tilde{g})^{-1}                                                         \\
		 & =\xi_S(g)\nu\xi_W(\tilde{g})\nu^{-1}\tilde{\xi}(\tilde{g})^{-1}                                                  \\
		 & =\xi_S(g)\gamma(\tilde{g}).
	\end{align}
	Thus, there exists \(\delta(\tilde{g})\in k^\times\) such that \(\gamma(\tilde{g})=\delta(\tilde{g})\Id_S\), so we get \(\delta(\tilde{g})\tilde{\xi}(\tilde{g})\nu=\nu\xi_W(\tilde{g})\) for any \(\tilde{g}\in\tilde{G}\).
	It is easy to show that \(\delta\colon \tilde{G}\rightarrow k^\times\) is a group homomorphism and induces the group homomorphism \(\bar{\delta}\colon\tilde{G}/G\rightarrow k^\times\).
	Let \(V\) be the \(k[\tilde{G}/G]\)-module corresponding to \(\bar{\delta}\). Then we have \(V\otimes \tilde{S}\cong W\).
\end{proof}
\begin{proposition}\label{brick extensions 2021-09-10 09:36:38}
	Let \(\tilde{G}\) be a finite group and \(G\) a normal subgroup of \(\tilde{G}\), \(S\) a semibrick in \(kG\lmod\) and \(S\cong \bigoplus_{i=1}^{n_{S}}S_i\) a direct sum decomposition into bricks. Assume the following conditions hold:
	\begin{itemize}
		\item \(S_i\) is \(\tilde{G}\)-stable for any \(i=1, \ldots, n_S\),
		\item \(H^2(\tilde{G}/G,k^\times)=1\),
		\item \(k[\tilde{G}/G]\) is basic as a \(k\)-algebra.
	\end{itemize}
	Then the following hold:
	\begin{enumerate}
		\item For each indecomposable direct summand \(S_i\) of \(S\), there exist exactly \(e:=|k[\tilde{G}/G]|\) isomorphism classes of extending \(k\tilde{G}\)-modules \(\tilde{S_i}^{(1)}, \ldots,\tilde{S_i}^{(e)}\) of \(S_i\).\label{tag extending number 2021-10-17 16:01:42}
		\item The direct sum \(\bigoplus_{i=1}^{n_S}\bigoplus_{j=1}^{e}\tilde{S_i}^{(j)}\) is a semibrick in \(k\tilde{G}\lmod\).\label{taag sum ext brick 2021-10-17 16:06:11}
	\end{enumerate}
\end{proposition}
\begin{proof}
	The assertion \ref{tag extending number 2021-10-17 16:01:42} immediately follows from \cref{extending action}.
	In order to prove the assertion \ref{taag sum ext brick 2021-10-17 16:06:11}, we only have to show that a direct sum \(\bigoplus_{j=1}^{e}\tilde{S_i}^{(j)}\) is also a semibrick by \cref{rem extending action hom}.
	Since the group algebra \(k[\tilde{G}/G]\) is basic, by \cref{extending action}, this module is isomorphic to
	\begin{equation}
		\bigoplus_{V} \left( \tilde{S_i}^{(1)} \otimes V \right) \cong \tilde{S_i}^{(1)} \otimes \Soc (k[\tilde{G}/G]),
	\end{equation}
	here \(V\) runs through representatives of the isomorphism classes of simple \(k[\tilde{G}/G]\)-modules.
	Now we have that
	\begin{align}
		\tilde{S_i}^{(1)} \otimes \Soc (k[\tilde{G}/G])
		 & \subset  \tilde{S_i}^{(1)} \otimes k[\tilde{G}/G] \\
		 & \cong    \induc_G^{\tilde{G}} S_i.
	\end{align}
	Hence, for any simple \(k[\tilde{G}/G]\)-module \(V\), we get that
	\begin{align}
		\Hom_{k\tilde{G}} (\tilde{S_i}^{(1)}\otimes  V,\tilde{S_i}^{(1)}\otimes V)\oplus
		 & \bigoplus_{V'\ncong V}\Hom_{k\tilde{G}} (\tilde{S_i}^{(1)}\otimes V,\tilde{S_i}^{(1)}\otimes V') \\
		 & \cong \Hom_{k\tilde{G}} (\tilde{S_i}^{(1)}\otimes V,\bigoplus_{V'}\tilde{S_i}^{(1)}\otimes V')   \\
		 & \subset \Hom_{k\tilde{G}} (\tilde{S_i}^{(1)}\otimes V,\induc_G^{\tilde{G}} S_i)                  \\
		 & \cong \Hom_{kG} (\restr_G^{\tilde{G}} (\tilde{S_i}^{(1)}\otimes V),S_i)                          \\
		 & \cong \Hom_{kG} (S_i,S_i).
	\end{align}
	Therefore, we get \(\bigoplus_{V\ncong V'}\Hom (\tilde{S_i}^{(1)}\otimes V,\tilde{S_i}^{(1)}\otimes V')=0\) since
	\begin{equation}
		\dim_k\Hom_{k\tilde{G}} (\tilde{S_i}^{(1)}\otimes V,\tilde{S_i}^{(1)}\otimes V)=\dim_k\Hom_{kG} (S_i,S_i)=1,
	\end{equation}
	which shows that \(\bigoplus_{V} ( \tilde{S_i}^{(1)} \otimes V)\cong \bigoplus_{j=1}^e \tilde{S_i}^{(j)}\) is a semibrick in \(k\tilde{G}\lmod\).
\end{proof}
\begin{remark}
	With the same assumption in \cref{brick extensions 2021-09-10 09:36:38}, if \(e\) is prime to \(p\), then we have
	\begin{equation}
		\induc_G^{\tilde{G}} S_i\cong \bigoplus_{j=1}^e \tilde{S}_i^{(j)}
	\end{equation}
	by \cref{rem ind and extending module 2021-09-29 08:59:16} and \cref{extending action}.
\end{remark}

\begin{theorem}\label{induc dsmmand number 2021-09-07 10:26:20}
	Assume that the cohomology group \(H^2(\tilde{G}/G,k^\times)\) is the trivial group and the group algebra \(k[\tilde{G}/G]\) is basic as a \(k\)-algebra. Then for any indecomposable \(\tilde{G}\)-stable \(kG\)-module \(U\), the number of indecomposable direct summand of \(\induc_G^{\tilde{G}} U\) is equal to \(|k[\tilde{G}/G]|\) and any two direct summands of \(\induc_G^{\tilde{G}} U\) are not isomorphic.
	In particular, we have that \(|\induc_G^{\tilde{G}} U|=|k[\tilde{G}/G]|\).
\end{theorem}
\begin{proof}
	The direct sum decompositions of the induced module \(\induc_G^{\tilde{G}}U\) into indecomposable \(k\tilde{G}\)-modules correspond to orthogonal primitive idempotents decompositions of the unit element of the endomorphism \(k\)-algebra \(\End_{k\tilde{G}}(\induc_G^{\tilde{G}} U)\).
	By the \(\tilde{G}\)-stability of \(U\), the algebra \(\End_{k\tilde{G}}(\induc_G^{\tilde{G}} U)\) has a strongly \(\tilde{G}/G\)-graded algebra structure with the two-sided ideal
	\begin{equation}\label{graded J rad 2021-09-28 04:08:53}
		\Rad(\End_{kG}(U))\End_{k\tilde{G}}(\induc_G^{\tilde{G}} U)
	\end{equation}
	by \cite[Lemma 4.6.5, Theorem 4.6.7]{MR998775}.
	Moreover, we have also the following inclusion of two-sided ideals:
	\begin{equation}
		\Rad(\End_{kG}(U))\End_{k\tilde{G}}(\induc_G^{\tilde{G}} U)\subset \Rad(\End_{k\tilde{G}}(\induc_G^{\tilde{G}} U)).
	\end{equation}
	In particular, the two-sided ideal \eqref{graded J rad 2021-09-28 04:08:53} is a nilpotent ideal.
	Therefore, by the idempotent lifting theorem (for example, see \cite[Theorem 1.4.10]{MR998775}), orthogonal primitive idempotents decompositions of unit element of \(\End_{k\tilde{G}}(\induc_G^{\tilde{G}} U)\) correspond to the ones of
	\begin{equation}\label{alginf}
		\End_{k\tilde{G}}(\induc_G^{\tilde{G}} U))/\Rad(\End_{kG}(U))\End_{k\tilde{G}}(\induc_G^{\tilde{G}} U).
	\end{equation}
	The algebra \(\End_{kG}(U)\) is local and \(k\) is an algebraically closed field, hence the algebra \eqref{alginf} isomorphic to a twisted group algebra \(k_\alpha[\tilde{G}/G]\), where \(\alpha\) is a Schur multiplier.
	However, the assumption \(H^2(\tilde{G}/G,k^\times)=1\) implies that \(k_\alpha[\tilde{G}/G]\) is isomorphic to the group algebra \(k[\tilde{G}/G]\) which is a basic \(k\)-algebra by our assumptions.
	Therefore, the number of indecomposable direct summand of \(\induc_G^{\tilde{G}} U\) is equal to \(|k[\tilde{G}/G]|\) and any two direct summands of \(\induc_G^{\tilde{G}} U\) are not isomorphic.
\end{proof}
\begin{lemma}\label{2021-09-09 18:06:40}
	Let \(G\) be a normal subgroup of a finite group \(\tilde{G}\) and \(V\) an indecomposable \(kG\)-module.
	If \(V\) is \(\tilde{G}\)-stable, then its projective cover \(P(V)\), syzygy \(\Omega(V)\), injective envelope \(I(V)\) and cosyzygy \(\Omega^{-1}(V)\) are also \(\tilde{G}\)-stable.
\end{lemma}
\begin{proof}
	For any \(\tilde{g}\in \tilde{G}\), there exists an isomorphism \(\phi\colon \tilde{g}V\rightarrow V\) by the \(\tilde{G}\)-stability of \(V\).
	We consider the following commutative diagram in \(kG\lmod\) with exact rows:
	\begin{equation}
		\begin{tikzcd}
			0\ar[r]&\tilde{g}\Omega (V) \ar[r] \ar[d,"\phi''"]& \tilde{g}P (V) \ar[r,"{}^{\tilde{g}}\mu"]\ar[d,"\phi'"]& \tilde{g}V\ar[r] \ar[d,"\phi"] & 0\\
			0\ar[r]&\Omega(V) \ar[r] & P(V) \ar[r,"\mu"] &  V\ar[r] &0.
		\end{tikzcd}
	\end{equation}
	Since \(\mu\) is an essential epimorphism and \(\phi\) is an isomorphism, the vertical morphisms \(\phi'\) and \(\phi''\) are isomorphisms.
	Moreover, by using the dual arguments, we get that \(I(V)\) and \(\Omega^{-1}(V)\) are also \(\tilde{G}\)-stable.
\end{proof}
\begin{lemma}\label{induction functor projective cover 2021-09-07 12:06:47}
	Let \(G\) be a normal subgroup of a finite group \(\tilde{G}\) satisfying that \(H^2(\tilde{G}/G,k^\times)=1\) and that \(k[\tilde{G}/G]\) is basic as a \(k\)-algebra.
	For any indecomposable \(\tilde{G}\)-stable \(kG\)-module \(V\), the following hold:
	\begin{enumerate}
		\item \(P(\induc_G^{\tilde{G}} V)\cong \induc_G^{\tilde{G}} P(V)\),\label{2020-03-25 01:17:41}
		\item \(\Omega (\induc_G^{\tilde{G}} V)\cong \induc_G^{\tilde{G}} \Omega(V)\),\label{2020-03-25 01:17:47}
		\item \(\tau(\induc_G^{\tilde{G}} V)\cong \induc_G^{\tilde{G}} \tau V\).\label{2020-03-25 01:17:54}
	\end{enumerate}
\end{lemma}
\begin{proof}
	We have the following commutative diagram in \(k\tilde{G}\lmod\) with exact rows:
	\begin{equation}
		\begin{tikzcd}
			0\ar[r]&\induc_G^{\tilde{G}}\Omega(V) \ar[r] \ar[d,"\varphi'"]& \induc_G^{\tilde{G}}P(V) \ar[r] \ar[d,"\varphi"]& \induc_G^{\tilde{G}} V\ar[r] \ar[d,"\mathrm{id}_{\induc_G^{\tilde{G}} V}"]&0\\
			0\ar[r]&\Omega (\induc_G^{\tilde{G}} V) \ar[r]& P(\induc_G^{\tilde{G}} V) \ar[r,"\mu"']&\induc_G^{\tilde{G}} V\ar[r]&0.
		\end{tikzcd}
	\end{equation}
	Since \(\mu\) is an essential epimorphism, the vertical morphisms \(\varphi\) and \(\varphi'\) are split epimorphisms,
	and there exists a direct summand \(Q\) of \(\induc_G^{\tilde{G}}P (V)\) such that \(Q\oplus P(\induc_G^{\tilde{G}}V)\cong \induc_G^{\tilde{G}}P (V)\) and \(Q\oplus \Omega (\induc_G^{\tilde{G}} V) \cong \induc_G^{\tilde{G}} \Omega (V)\). We recall that \(\Omega V\) is indecomposable and \(\tilde{G}\)-stable by \cref{2021-09-09 18:06:40}.
	The numbers of indecomposable direct summands of \(\induc_G^{\tilde{G}} \Omega(V)\) and \(\Omega (\induc_G^{\tilde{G}} V)\) are equal to \(|k[\tilde{G}/G]|\) by \cref{induc dsmmand number 2021-09-07 10:26:20}.
	Hence, we get \(Q=0\) and the proofs of assertion \ref{2020-03-25 01:17:41} and assertion \ref{2020-03-25 01:17:47} are completed.

	Finally, we prove the assertion \ref{2020-03-25 01:17:54}.
	Since \(k\tilde{G}\) and \(kG\) are symmetric \(k\)-algebras, it holds that \(\tau V\cong \Omega \Omega(V)\) and \(\tau(\induc_G^{\tilde{G}} V)\cong \Omega \Omega (\induc_G^{\tilde{G}} V)\) for any \(kG\)-module \(V\).
	Therefore, we complete the proof of assertion \ref{2020-03-25 01:17:54}.
\end{proof}
The dual argument gives the following.
\begin{lemma}\label{induction functor inj env 2021-09-07 12:06:47}
	Let \(G\) be a normal subgroup of a finite group \(\tilde{G}\) satisfying that \(H^2(\tilde{G}/G,k^\times)=1\) and that \(k[\tilde{G}/G]\) is basic as a \(k\)-algebra.
	For any indecomposable \(\tilde{G}\)-stable \(kG\)-module \(V\), the following hold:
	\begin{enumerate}
		\item \(I(\induc_G^{\tilde{G}} V)\cong \induc_G^{\tilde{G}} I (V)\),\label{2021-10-01 15:13:53}
		\item \(\Omega^{-1} (\induc_G^{\tilde{G}} V)\cong \induc_G^{\tilde{G}} \Omega^{-1} (V)\),\label{2021-10-01 15:13:57}
		\item \(\tau^{-1}(\induc_G^{\tilde{G}} V)\cong \induc_G^{\tilde{G}} \tau^{-1} V\).\label{2021-10-01 15:14:01}
	\end{enumerate}
\end{lemma}

\subsection{Blocks of group algebras}

We recall the definition of blocks of group algebras. Let \(G\) be a finite group. The group algebra \(kG\) has a unique decomposition
\begin{equation}\label{block dec 2021-11-10 11:00:13}
	kG=B_0\times \cdots \times B_l
\end{equation}
into the direct product of indecomposable \(k\)-algebras \(B_i\).
We call each indecomposable direct product component \(B_i\) a block of \(kG\) and the above decomposition the block decomposition. We remark that any block \(B_i\) is a two-sided ideal of \(kG\).

For any indecomposable \(kG\)-module \(U\), there exists a unique block \(B_i\) of \(kG\) such that \(U=B_iU\) and \(B_jU=0\) for all \(j\neq i\). Then we say that \(U\) lies in the block \(B_i\) or simply \(U\) is a \(B_i\)-module.
We denote by \(B_0(kG)\) the principal block of \(kG\), in which the trivial \(kG\)-module lies.

\begin{remark}
	We remark that the block decomposition \eqref{block dec 2021-11-10 11:00:13} induces the following isomorphism of partially ordered sets:
	\begin{equation}
		\begin{tikzcd}[row sep=1pt]
			\stautilt(kG)\ar[r]& \stautilt(B_0)\times\cdots\times\stautilt(B_l)\\
			M\ar[r,mapsto]&(B_0 M, \ldots, B_l M).
		\end{tikzcd}
	\end{equation}
\end{remark}

Now we recall the definition and basic properties of defect groups of blocks.
\begin{definition}
	Let \(B\) be a block of \(kG\).
	A defect group \(D\) of \(B\) is a minimal subgroup of \(G\) satisfying the following condition:
	the \(B\)-bimodule morphism
	\begin{equation}
		\begin{tikzcd}[row sep=1pt]
			B\otimes_{kD}B \ar[r,"\mu_D"]& B\\
			\beta _1\otimes \beta _2 \ar[r,mapsto]&  \beta _1 \beta _2
		\end{tikzcd}
	\end{equation}
	is a split epimorphism.
\end{definition}

\begin{proposition}[{see \cite[Chapter 4, 5]{MR860771}}]
	Let \(B\) be a block of \(kG\) and \(D\) a defect group of \(B\). Then the following hold:
	\begin{enumerate}
		\item \(D\) is a \(p\)-subgroup of \(G\) and the set of all defect groups of \(B\) forms the conjugacy class of \(D\) in \(G\).
		\item \(D\) is a cyclic group if and only if the algebra \(B\) is finite representation type.
		\item If \(B\) is the principal block of \(kG\), then \(D\) is a Sylow \(p\)-subgroup of \(G\).
	\end{enumerate}
\end{proposition}

\begin{theorem}[{see \cite[Corollary 14.6, Theorem 17.1 and proof of Lemma 19.3]{MR860771}}]\label{Theorem:Dade Brauer Tree}
	Let \(B\) be a block of \(kG\) and \(D\) a defect group of \(B\).
	\begin{enumerate}
		\item \(D\) is the trivial group if and only if \(B\) is a simple algebra.
		\item \(D\) is a non-trivial cyclic group if and only if \(B\) is a Brauer tree algebra with \(e\) edges and  multiplicity \((|D|-1)/e\), where \(e\) is a devisor of \(p-1\).\label{2021-12-06 18:37:08}
	\end{enumerate}
\end{theorem}
\subsection{Clifford's theory for blocks of normal subgroups}
Let \(G\) be a finite group, \(\tilde{G}\) a finite group containing \(G\) as a normal subgroup, \(B\) a block of \(kG\) and \(\tilde{B}\) a block of \(k\tilde{G}\).
We say that \(\tilde{B}\) covers \(B\) or that \(B\) is covered by \(\tilde{B}\) if \(1_B 1_{\tilde{B}}\neq 0\). We denote by \(\inertiagp_{\tilde{G}}(B)\) the  inertial group of \(B\) in \(\tilde{G}\), that is \(\inertiagp_{\tilde{G}}(B):=\left\{ x \in \tilde{G} \setmid xBx^{-1} = B \right\}\).

\begin{remark}[{see \cite[Theorem 15.1, Lemma 15.3]{MR860771}}]\label{Remark:cover}
	With the above notation, the following are equivalent:
	\begin{enumerate}
		\item The block \(\tilde{B}\) covers \(B\).
		\item There exists a non-zero \(\tilde{B}\)-module \(U\) such that \(\restr_G^{\tilde{G}} U\) has a non-zero direct summand lying in \(B\).
		\item For any non-zero \(\tilde{B}\)-module \(U\), there exists a non-zero direct summand of \(\restr_G^{\tilde{G}} U\) lying in \(B\).
	\end{enumerate}
\end{remark}

\begin{remark}\label{cover principal 2022-09-05 14:33:21}
	The principal block \(B_0(kG)\) of \(kG\) is covered by the principal block \(B_0(k\tilde{G})\) of \(k\tilde{G}\) and \(\inertiagp_{\tilde{G}}(B_0(kG))=\tilde{G}\) since the trivial \(kG\)-module \(k_G\) is \(\tilde{G}\)-stable \(\restr_G^{\tilde{G}} k_{\tilde{G}} \cong k_G\).
\end{remark}

\begin{theorem}[{Clifford's Theorem for blocks \cite[Theorem 15.1, Lemma 15.3]{MR860771}}]\label{Clifford's Theorem block ver}
	Let \(\tilde{G}\) be a finite group, \(G\) a normal subgroup of \(\tilde{G}\), \(B\) a block of \(kG\), \(\tilde{B}\) a block of \(k\tilde{G}\) covering \(B\) and \(U\) a \(\tilde{B}\)-module. Then the following hold:
	\begin{enumerate}
		\item The set of blocks of \(kG\) covered by \(\tilde{B}\) equals to the conjugacy class of \(B\) in \(\tilde{G}\):
		      \begin{equation}
			      \left\{ B' \setmid \text{\(B'\) is a block of \(kG\) covered by \(\tilde{B}\)} \right\}=\left\{ xBx^{-1} \setmid x \in \tilde{G} \right\}.
		      \end{equation}
		\item We get the following isomorphism of \(kG\)-modules:
		      \begin{equation}
			      \restr_G^{\tilde{G}} U \cong \bigoplus_{x \in [\tilde{G}/\inertiagp_{\tilde{G}}(B)]}xBU.
		      \end{equation}
	\end{enumerate}
\end{theorem}

\begin{proposition}[{see \cite[Theorem 5.5.10 Theorem 5.5.12]{MR998775}}]\label{Morita equivalence covering block}
	Let \(G\) be a normal subgroup of a finite group \(\tilde{G}\), \(B\) a block of \(kG\) and \(\beta\) a block of \(k\inertiagp_{\tilde{G}}(B)\) covering \(B\).
	Then the following hold:
	\begin{enumerate}
		\item For any \(B\)-module \(V\), the induced module \(\induc_G^{\inertiagp_{\tilde{G}}(B)}V\) is a direct sum of \(k\inertiagp_{\tilde{G}}(B)\)-module lying blocks covering \(B\).
		\item There exists a unique block \(\tilde{B}\) of \(k\tilde{G}\) covering \(B\) such that the induction functor \(\induc_{\inertiagp_{\tilde{G}}(B)}^{\tilde{G}} \colon k\inertiagp_{\tilde{G}}(B)\lmod \rightarrow k{\tilde{G}}\lmod\) restricts to a Morita equivalence
		      \begin{equation}\label{2020-03-25 15:17:09}
			      \induc_{\inertiagp_{\tilde{G}}(B)}^{\tilde{G}} \colon \beta\lmod \rightarrow \tilde{B}\lmod.
		      \end{equation}
	\end{enumerate}
\end{proposition}

\begin{proposition}[{\cite[Corollary 5.5.6, Theorem 5.5.13, Lemma 5.5.14]{MR998775}}]\label{p-power index covering uniqueness}
	Let \(G\) be a normal subgroup of \(\tilde{G}\) and \(B\) a block of \(kG\), then the following conditions hold:
	\begin{enumerate}
		\item If \(\tilde{G}/G\) is a \(p\)-group, then there exists a unique block of \(k\tilde{G}\) covering \(B\).
		\item If a defect group \(D\) of \(B\) satisfies \(C_{\tilde{G}}(D)\subset G\), then  there exists a unique block of \(k\tilde{G}\) covering \(B\).
	\end{enumerate}
\end{proposition}

\begin{lemma}\label{counting 2021-09-10 17:42:55}
	Let \(G\) be a normal subgroup of a finite group  \(\tilde{G}\) satisfying that \(H^2(\tilde{G}/G,k^\times)=1\) and that \(k[\tilde{G}/G]\) is basic as a \(k\)-algebra, \(B\) a \(\tilde{G}\)-stable block of \(kG\) and \(A\) the direct product of the all blocks of \(k\tilde{G}\) covering \(B\). Assume that any simple \(B\)-module is \(\tilde{G}\)-stable. Then we have \(|A|=|k[\tilde{G}/G]|\cdot|B|\).
\end{lemma}
\begin{proof}
	By \cref{rem extending action hom} and \cref{Remark:cover},
	any extending \(k\tilde{G}\)-module of a simple \(B\)-module is also a simple \(A\)-module, and we get that \(|A|\geq|k[\tilde{G}/G]|\cdot|B|\).
	On the other hand,
	for an arbitrary simple \(A\)-module \(S'\), there exists a simple \(B\)-submodule \(S\) of \(\restr_G^{\tilde{G}}S'\) such that \(\restr_G^{\tilde{G}}S'\cong S^{\oplus r}\) for some positive integer \(r\) by \cref{Clifford's Theorem of simple module} and \cref{Remark:cover}.
	We denote an extending \(k\tilde{G}\)-module of \(S\) by \(\tilde{S}\), which is also a simple \(A\)-module.
	By \cref{rem ind and extending module 2021-09-29 08:59:16}, we get
	\begin{align}
		S'\subset \induc_G^{\tilde{G}}\restr_G^{\tilde{G}} S'
		 & \cong\induc_G^{\tilde{G}}S^{\oplus r}               \\
		 & \cong (\tilde{S}\otimes k[\tilde{G}/G])^{\oplus r}.
	\end{align}
	Therefore, we get a \(1\)-dimensional \(k[\tilde{G}/G]\)-module \(V\) such that \(S'\cong \tilde{S}\otimes V\) by the Jordan--H\"{o}lder theorem.
	Therefore, \(S'\) is also an extending \(k\tilde{G}\)-module of \(S\).
	Hence, we get that \(|A|\leq|k[\tilde{G}/G]|\cdot|B|\) by \cref{extending action}.
\end{proof}

\section{The main results and their applications}
In this section, we give lemmas and prove the main results.
After that, we give some applications and examples of our main results.
\subsection{Main theorems and their proof}\label{subsec main results 2022-09-07 16:40:50}
The following assertions are related to the assumptions of our main results.
\begin{lemma}\label{stable bricks conditions 2021-09-30 17:16:10}
	Let \(G\) be a normal subgroup of \(\tilde{G}\) and \(B\) a \(\tilde{G}\)-stable block of \(kG\). Then the following are equivalent:
	\begin{enumerate}
		\item Any left finite brick in \(B\lmod\) is \(\tilde{G}\)-stable.\label{stable item 2021-09-07 08:27:11}
		\item Any indecomposable \(\tau\)-rigid \(B\)-module is \(\tilde{G}\)-stable.\label{stable item 2021-09-07 08:27:25}
		\item Any right finite brick in \(B\lmod\) is \(\tilde{G}\)-stable.\label{stable item 2021-09-07 08:27:39}
		\item Any indecomposable \(\tau^{-1}\)-rigid \(B\)-module is \(\tilde{G}\)-stable.\label{stable item 2021-09-07 08:27:49}
	\end{enumerate}
\end{lemma}
\begin{proof}
	Let \(S\) be a right finite brick over \(B\) and \(U\) the indecomposable \(\tau\)-rigid module corresponding to \(S\) in the bijection \eqref{2021-09-22 14:16:57}.
	We can easily check that \(\tilde{g}\Fac(U)=\Fac(\tilde{g}U)\) and \(\tilde{g}\torscl(S)=\torscl(\tilde{g}S)\) for any \(\tilde{g}\in \tilde{G}\), which implies the equivalence of \ref{stable item 2021-09-07 08:27:11} and \ref{stable item 2021-09-07 08:27:25} by \cref{asai correspondence}.
	The similar argument show the equivalence of \ref{stable item 2021-09-07 08:27:39} and \ref{stable item 2021-09-07 08:27:49}.
	We prove the equivalence of \ref{stable item 2021-09-07 08:27:25} and \ref{stable item 2021-09-07 08:27:49}.
	For any indecomposable non-projective module \(V\), the \(\tilde{G}\)-stability of \(V\) means those of \(\Omega V\), \(\Omega^{-1} V\), \(\tau V\) and \(\tau^{-1} V\) by \cref{2021-09-09 18:06:40}.
	Hence, the conclusion follows from \cref{indec tau taui corr2021-09-09 16:44:38}.
\end{proof}

The following is a slight generalization of \cite[Lemma 3.22]{MR4243358}.
\begin{lemma}[{\cite[Lemma 3.22]{MR4243358}}]\label{Cyclicdefect p-power index 2022-08-28 21:45:58}
	Let \(G\) be a normal subgroup of a finite group \(\tilde{G}\)  and \(B\) a \(\tilde{G}\)-stable block of \(kG\) with a cyclic defect group. Then the following hold:
	\begin{enumerate}
		\item If any simple \(B\)-module is \(\tilde{G}\)-stable, then any indecomposable \(B\)-module is also \(\tilde{G}\)-stable.\label{2021-12-06 18:35:10}
		\item If \(\tilde{G}/G\) is a \(p\)-group, then any indecomposable \(B\)-module is \(\tilde{G}\)-stable.\label{2021-12-06 18:35:40}
	\end{enumerate}
\end{lemma}
\begin{proof}
	We prove the assertion \ref{2021-12-06 18:35:10} by using similar way as \cite[Lemma 3.22]{MR4243358}.
	We prove that \(\decompgp_{\tilde{G}}(V)=\tilde{G}\) for indecomposable \(B\)-module \(V\) by induction on the composition length of \(V\).
	If \(V\) is simple or indecomposable projective, there is nothing to show by the assumption and \cref{2021-09-09 18:06:40}.
	We assume that the composition length of \(V\) is two or more and that \(V\) is not projective.
	We remark that any indecomposable non-projective \(B\)-module is a string module (for example, see \cite{MR3823391}).
	Hence, we can take a simple \(B\)-module \(S\) and an indecomposable \(B\)-module \(V'\) which satisfy at least one of the following conditions:
	\begin{enumerate}
		\item There exists a non-split exact sequence
		      \begin{equation}
			      \begin{tikzcd}
				      0 \ar[r] & S \ar[r,"\mu"] & V \ar[r,"\nu"] & V'  \ar[r] & 0.
			      \end{tikzcd}
		      \end{equation}
		\item There exists a non-split exact sequence
		      \begin{equation}
			      \begin{tikzcd}
				      0 \ar[r] & V' \ar[r,"\mu'"] & V \ar[r,"\nu'"] & S \ar[r] & 0.
			      \end{tikzcd}
		      \end{equation}
	\end{enumerate}
	It suffices to prove \(\decompgp_{\tilde{G}}(V)=\tilde{G}\) under the assumption that there exists the first exact sequence;
	the other case can be proved similarly.
	Since \(\Ext_B^1(V', S)\) is \(1\)-dimensional over \(k\) (see \cite[Proposition 21.7]{MR860771}),
	we can prove the conclusion by induction on the composition length of \(V\).

	By \cref{Theorem:Dade Brauer Tree} and \cite[Lemma 2.2]{MR2592757},
	any simple \(B\)-module \(S\) is \(\tilde{G}\)-stable,
	we have the assertion \ref{2021-12-06 18:35:40} from the first one.
\end{proof}

Now we give proofs of the main theorems. First, we state the main result again, which are stated in the introduction.
\begin{theorem}\label{MT 2021-09-07 13:50:56}
	Let \(\tilde{G}\) be a finite group, \(G\) a normal subgroup of \(\tilde{G}\), \(B\) be a block of \(kG\) and \(\tilde{B}\) be a block of \(k\tilde{G}\) covering \(B\).
	We assume the following conditions hold:
	\begin{itemize}
		\item Any left finite bricks in \(B\lmod\) is \(\inertiagp_{\tilde G}(B)\)-stable.
		\item \(H^2(\inertiagp_{\tilde{G}}(B)/G,k^\times)=1\).
		\item \(k[\inertiagp_{\tilde{G}}(B)/G]\) is basic as a \(k\)-algebra.
	\end{itemize}
	Then the maps
	\begin{equation}\label{stau corr 2021-09-07 13:51:09}
		\begin{tikzcd}[ampersand replacement=\&,row sep=1pt]
			\stautilt B \ar[r]\&\stautilt \tilde{B}\\
		\end{tikzcd}
	\end{equation}
	defined by \(\stautilt B\ni M\mapsto \tilde{B}\induc_G^{\tilde{G}} M \in \stautilt \tilde{B}\) and
	\begin{equation}\label{twotilt corr 2021-09-10 17:53:30}
		\begin{tikzcd}[ampersand replacement=\&,row sep=1pt]
			\twotilt B \ar[r]\&\twotilt \tilde{B}\\
		\end{tikzcd}
	\end{equation}
	defined by \(\twotilt B\ni T\mapsto \tilde{B}\induc_G^{\tilde{G}} T \in \twotilt \tilde{B}\) are well-defined and injective.
	Moreover, we have the following commutative diagram:
	\begin{equation}\label{cd 2021-12-16 20:06:39}
		\begin{tikzcd}[ampersand replacement=\&, column sep=2cm]
			\stautilt B\ar[d,"\text{\eqref{2siltcorr} for \(B\)}"',"\wr"] \ar[r,"\text{\eqref{stau corr 2021-09-07 13:51:09}}"]\&\stautilt \tilde{B} \ar[d,"\text{\eqref{2siltcorr} for \(\tilde{B}\)}","\wr"']\\
			\twotilt B\ar[r,"\text{\eqref{twotilt corr 2021-09-10 17:53:30}}"']\&\twotilt \tilde{B}.
		\end{tikzcd}
	\end{equation}
\end{theorem}
\begin{proof}
	By \cref{Theorem: Frobenius and projective} and \cref{Morita equivalence covering block}, we may assume that \(B\) is \(\tilde{G}\)-stable, that is, \(\tilde{G}=\inertiagp_{\tilde{G}}(B)\).
	For \(M \in \stautilt B\), let
	\begin{equation}
		P_1\xrightarrow{f_1}P_0\xrightarrow{f_0}M\rightarrow 0
	\end{equation}
	be a minimal projective presentation of \(M\) and \(P\) a projective module such that \(\Hom_{B}(P,M)=0\) and \(|P|+|M|=|B|\).
	Let \(A\) be the direct product algebra of all blocks of \(k\tilde{G}\) covering \(B\).
	In order to show that the maps \eqref{stau corr 2021-09-07 13:51:09} and \eqref{twotilt corr 2021-09-10 17:53:30} are well-defined and that the diagram \eqref{cd 2021-12-16 20:06:39} is commutative, we only have to show the following:
	\begin{enumerate}
		\item \(\induc_G^{\tilde{G}} M\) is a \(\tau\)-rigid \(k\tilde{G}\)-module.\label{item tau rigid 2021-10-25 15:28:38}
		\item \(\Hom_A(\induc_G^{\tilde{G}} P,\induc_G^{\tilde{G}} M)=0\).\label{item pair 2021-10-25 15:35:12}
		\item \(|\induc_G^{\tilde{G}} M|+|\induc_G^{\tilde{G}} P|=|A|\).\label{item count pair 2021-10-25 15:45:36}
		\item The sequence
		      \begin{equation}
			      \induc_G^{\tilde{G}}P_1\xrightarrow{\induc_G^{\tilde{G}}f_1}\induc_G^{\tilde{G}}P_0\xrightarrow{\induc_G^{\tilde{G}}f_0}\induc_G^{\tilde{G}}M\rightarrow 0
		      \end{equation}
		      is also a minimal projective presentation of \(k\tilde{G}\)-module \(\induc_G^{\tilde{G}}M\).\label{item minimal presentation 2021-10-25 16:02:34}
	\end{enumerate}
	We show \ref{item tau rigid 2021-10-25 15:28:38}, that is, \(\induc_G^{\tilde{G}} M\) is also a \(\tau\)-rigid \(k\tilde{G}\)-module as follows:

	\begin{align}
		\Hom_{k\tilde{G}}(\induc_G^{\tilde{G}}M,\tau\induc_G^{\tilde{G}}M)
		 & \cong \Hom_{k\tilde{G}}(\induc_G^{\tilde{G}}M,\induc_G^{\tilde{G}}\tau M) \\
		 & \cong \Hom_{kG}(\restr_G^{\tilde{G}}\induc_G^{\tilde{G}}M,\tau M)         \\
		 & \cong \Hom_{kG}(\bigoplus_{x\in[\tilde{G}/G]}xM,\tau M)                   \\
		 & \cong \bigoplus_{x\in[\tilde{G}/G]}\Hom_{B}(M,\tau M)                     \\
		 & = 0,
	\end{align}
	here the first isomorphism comes from \cref{induction functor projective cover 2021-09-07 12:06:47}, the second isomorphism comes from \cref{Theorem: Frobenius and projective}, the third isomorphism comes from \cref{Mackey's decomposition formula} and the fourth isomorphism comes from our assumption. We show \ref{item pair 2021-10-25 15:35:12} as follows:
	\begin{align}
		\Hom_{A}(\induc_G^{\tilde{G}} P,\induc_G^{\tilde{G}} M)
		 & \cong \Hom_{kG}(\restr_G^{\tilde{G}} \induc_G^{\tilde{G}} P,M) \\
		 & \cong \Hom_{kG}(\bigoplus_{x\in[\tilde{G}/G]}x P,M)            \\
		 & \cong \bigoplus_{x\in[\tilde{G}/G]}\Hom_{B}( P,M)              \\
		 & \cong 0,
	\end{align}
	here the first isomorphism comes from \cref{Theorem: Frobenius and projective}, the second isomorphism comes from \cref{Mackey's decomposition formula} and the third isomorphism comes from our assumption.
	The assertion \ref{item count pair 2021-10-25 15:45:36} is followed by \cref{counting 2021-09-10 17:42:55}.
	By \cref{tau number remark 2021-09-07 12:11:53}, we have that the map \eqref{stau corr 2021-09-07 13:51:09} is well-defined.
	The assertion \ref{item minimal presentation 2021-10-25 16:02:34} is followed by \cref{induction functor projective cover 2021-09-07 12:06:47}.
	Therefore, we get the map \eqref{twotilt corr 2021-09-10 17:53:30} is well-defined, and the diagram \eqref{cd 2021-12-16 20:06:39} is commutative.
	The injectivities of \eqref{stau corr 2021-09-07 13:51:09} and \eqref{twotilt corr 2021-09-10 17:53:30} are obvious.
\end{proof}

\begin{theorem}\label{MT sbrick 2021-09-10 09:27:04}
	With the same assumptions in \cref{MT 2021-09-07 13:50:56}, the following hold:
	\begin{enumerate}
		\item For any left finite semibrick \(S\) in \(B\lmod\) and its indecomposable direct summand \(S_i\) of \(S\), there exist exactly \(e:=|k[\inertiagp_{\tilde{G}}(B)/G]|\) isomorphism classes of extending \(k\inertiagp_{\tilde{G}}(B)\)-modules \(\tilde{S_i}^{(1)}, \ldots, \tilde{S_i}^{(e)}\) of \(S_i\).
		\item Then the map
		      \begin{equation}\label{sbrick corr 2021-09-10 09:27:09}
			      \begin{tikzcd}[ampersand replacement=\&,row sep=1pt]
				      \flsbrick B \ar[r]\&\flsbrick \tilde{B}\\
			      \end{tikzcd}
		      \end{equation}
		      defined by \(S\cong \bigoplus_{i=1}^{n_{S}}S_i\mapsto \tilde{B}\induc^{\tilde{G}}_{\inertiagp_{\tilde{G}}(B)}\left( \bigoplus_{i=1}^{n_S}\bigoplus_{j=1}^{e}\tilde{S}_i^{(j)} \right)\) is well-defined and injective, here \(S\cong \bigoplus_{i=1}^{n_{S}}S_i\) is a direct sum decomposition into bricks.
		\item We get the following  commutative diagram:
		      \begin{equation}
			      \begin{tikzcd}[ampersand replacement=\&, column sep=2cm]
				      \stautilt B\ar[d,"\text{\eqref{asaicorresp} for \(B\)}"',"\wr"] \ar[r,"\text{\eqref{stau corr 2021-09-07 13:51:09}}"]\&\stautilt \tilde{B} \ar[d,"\text{\eqref{asaicorresp} for \(\tilde{B}\)}","\wr"']\\
				      \flsbrick B\ar[r,"\text{\eqref{sbrick corr 2021-09-10 09:27:09}}"']\&\flsbrick \tilde{B}.
			      \end{tikzcd}
		      \end{equation}\label{cmuutative diangram 2021-09-10 09:50:21}
	\end{enumerate}
\end{theorem}
\begin{proof}
	By \cref{Theorem: Frobenius and projective} and \cref{Morita equivalence covering block}, we may assume that \(B\) is \(\tilde{G}\)-stable.
	In order to apply \cref{brick extensions 2021-09-10 09:36:38}, we show that any brick \(S_i\) appearing as a direct summand of the left finite semibrick \(S\) in \(B\lmod\) is \(\tilde{G}\)-stable.
	Let \(M\) be the support \(\tau\)-tilting \(B\)-module corresponding to \(S\) under \eqref{asaicorresp} and \(M'\) the support \(\tau\)-tilting \(B\)-module corresponding to \(S_i\) under \eqref{2021-09-19 23:16:42} for \(M\).
	By \cref{labeling properties 2021-11-14 17:10:59}, the subcategory \(\Fac M\cap \Sub N'\) contains a unique brick, which isomorphic to \(S_i\), where \(N'\) is the support \(\tau^{-1}\)-tilting \(B\)-module corresponding to \(M'\) under \eqref{ddual}.
	For any \(\tilde{g}\in \tilde{G}\), we have \begin{equation}\label{subcat eq 2021-11-18 12:12:13}
		\tilde{g}(\Fac M\cap \Sub N')=\Fac M\cap \Sub N'
	\end{equation}
	by \cref{stable bricks conditions 2021-09-30 17:16:10}.
	On the other hand, the subcategory in left-hand side of \eqref{subcat eq 2021-11-18 12:12:13} contains the brick \(\tilde{g}S\), which means that the subcategory \(\Fac M\cap \Sub N'\) contains \(\tilde{g}S\).
	Hence, \(\tilde{g} S \cong S\) by the uniqueness of the brick in \(\Fac M\cap \Sub N'\).
	Therefore, we have shown the well-definedness of \eqref{sbrick corr 2021-09-10 09:27:09}. The injectivity of \eqref{sbrick corr 2021-09-10 09:27:09} is obvious.

	Let \(M\) and \(S\) be corresponding support \(\tau\)-tilting \(B\)-module and semibrick in \(B\lmod\) under \eqref{asaicorresp} for \(B\), that is, it holds \(\Fac M=\torscl(S)\).
	To show that the diagram in \ref{cmuutative diangram 2021-09-10 09:50:21} is commutative,
	we only have to show that
	\(\Fac (\induc_G^{\tilde{G}}M)=\torscl(\induc_G^{\tilde{G}}S)\) by \cref{rem ind and extending module 2021-09-29 08:59:16} because we know that \(\induc_G^{\tilde{G}}M\) is a support \(\tau\)-tilting module by \cref{MT 2021-09-07 13:50:56} and because \(\bigoplus_{i=1}^{n_S}\bigoplus_{j=1}^{e}\tilde{S}_i^{(j)}\) is a semibrick by \cref{brick extensions 2021-09-10 09:36:38}.

	Since \(M/R(M,M)\cong S\), \(\induc_G^{\tilde{G}} S\) is a homomorphic image of \(\induc_G^{\tilde{G}}M\) by the exactness of \(\induc_G^{\tilde{G}}\).
	Therefore, \(\Fac (\induc_G^{\tilde{G}}M)\supset\torscl(\induc_G^{\tilde{G}}S)\) holds.
	We show the reverse inclusion.
	Since \(M \in \torscl(S)=\Filt (\Fac (S))\), we get a filtration of \(M\)
	\begin{equation}
		0=M_0\subset M_1\subset \cdots\subset M_{l-1}\subset M_l=V,
	\end{equation}
	such that \(M_{i+1}/M_i\in \Fac(S)\).
	By the exactness of \(\induc_G^{\tilde{G}}\) again, we get a filtration
	\begin{equation}
		0=\induc_G^{\tilde{G}}M_0\subset \induc_G^{\tilde{G}}M_1\subset \cdots\subset \induc_G^{\tilde{G}}M_{l-1}\subset \induc_G^{\tilde{G}}M_l=\induc_G^{\tilde{G}}V
	\end{equation}
	of \(\induc_G^{\tilde{G}} M\) such that
	\(\induc_G^{\tilde{G}}M_{i+1}/\induc_G^{\tilde{G}}M_i\in \Fac(\induc_G^{\tilde{G}}S)\).
	Therefore, we get that \(\Fac (\induc_G^{\tilde{G}}M)\subset\Fac(\Filt(\induc_G^{\tilde{G}}S))=\torscl(\induc_G^{\tilde{G}}S)\).
\end{proof}

We can prove the following theorem by the dual argument on \cref{MT 2021-09-07 13:50:56,MT sbrick 2021-09-10 09:27:04}.
\begin{theorem}\label{MT sbrick 2021-09-10 18:03:52}
	With the same assumptions in \cref{MT 2021-09-07 13:50:56}, the following hold:
	\begin{enumerate}
		\item The map
		      \begin{equation}\label{stau corr 2021-09-10 18:14:33}
			      \begin{tikzcd}[ampersand replacement=\&,row sep=1pt]
				      \stauitilt B \ar[r]\&\stauitilt \tilde{B}\\
			      \end{tikzcd}
		      \end{equation}
		      defined by \(\stauitilt B\ni N\mapsto \tilde{B}\induc_G^{\tilde{G}} N \in \stauitilt \tilde{B}\) is well-defined and injective.
		\item For any right finite brick \(S_i\) in \(B\lmod\), there exist exactly \(e:=|k[\tilde{G}/G]|\) isomorphism classes of extending \(k\tilde{G}\)-modules \(\tilde{S_i}^{(1)}, \ldots, \tilde{S_i}^{(e)}\) of \(S_i\).
		\item Then the map
		      \begin{equation}\label{sbrick corr 2021-09-10 18:10:41}
			      \begin{tikzcd}[ampersand replacement=\&,row sep=1pt]
				      \frsbrick B \ar[r]\&\frsbrick \tilde{B}\\
			      \end{tikzcd}
		      \end{equation}
		      defined by \(S\cong \bigoplus_{i=1}^{n_{S}}S_i\mapsto \tilde{B} \induc^{\tilde{G}}_{\inertiagp_{\tilde{G}}(B)}\left( \bigoplus_{i=1}^{n_S}\bigoplus_{j=1}^{e}\tilde{S}_i^{(j)} \right)\) is well-defined and injective, here \(S\cong \bigoplus_{i=1}^{n_{S}}S_i\) is a direct sum decomposition into bricks.
		\item We get the following  commutative diagram:
		      \begin{equation}
			      \begin{tikzcd}[ampersand replacement=\&]
				      \stauitilt B\ar[d,"\text{\eqref{asaicorresp} for \(B\)}"',"\wr"] \ar[r,"\text{\eqref{stau corr 2021-09-10 18:14:33}}"]\&\stauitilt \tilde{B} \ar[d,"\text{\eqref{asaicorresp} for \(\tilde{B}\)}","\wr"']\\
				      \frsbrick B\ar[r,"\text{\eqref{sbrick corr 2021-09-10 18:10:41}}"']\&\frsbrick \tilde{B}.
			      \end{tikzcd}
		      \end{equation}
	\end{enumerate}
\end{theorem}

The next theorem is followed by \cref{MT sbrick 2021-09-10 09:27:04} and \cref{MT sbrick 2021-09-10 18:03:52}.

\begin{theorem} With the same assumptions in \cref{MT 2021-09-07 13:50:56}, the following diagram is commutative:
	\begin{equation}
		\begin{tikzcd}
			&\frsbrick B\ar[dd,leftarrow,"\text{\eqref{dualbrick} for \(B\)}"'very near end]  \ar[rr,"\text{\eqref{sbrick corr 2021-09-10 18:10:41}}"]&&\frsbrick \tilde{B}\\
			\stauitilt B\ar[ru,"\text{\eqref{dasaicorresp 2021-09-24 13:16:37} for B}"description] \ar[rr,"\text{\eqref{stau corr 2021-09-10 18:14:33}}"near end,crossing over]&&\stauitilt \tilde{B}\ar[ru,"\text{\eqref{dasaicorresp 2021-09-24 13:16:37} for \(\tilde{B}\)}"description] \\
			&\flsbrick B\ar[rr,"\eqref{sbrick corr 2021-09-10 09:27:09}"very near start]&&\flsbrick \tilde{B}\ar[uu,"\text{\eqref{dualsbrick} for \(\tilde{B}\)}"'].\\
			\stautilt B\ar[ru,"\text{\eqref{asaicorresp} for \(B\)}"description] \ar[uu,"\text{\eqref{ddual} for \(B\)}"] \ar[rr,"\text{\eqref{stau corr 2021-09-07 13:51:09}}"'] &&\stautilt \tilde{B} \ar[ru,"\text{\eqref{asaicorresp} for \(\tilde{B}\)}"description] \ar[uu,"\text{\eqref{ddual} for \(\tilde{B}\)}"'very near end,crossing over]
		\end{tikzcd}
	\end{equation}
\end{theorem}
\begin{proof}
	By \cref{comm dbrick 2021-10-01 20:08:26}, \cref{MT sbrick 2021-09-10 09:27:04} and \cref{MT sbrick 2021-09-10 18:03:52}, we only have to show that the following diagram is commutative:
	\begin{equation}
		\begin{tikzcd}
			\stauitilt B\ar[r,"\eqref{stau corr 2021-09-10 18:14:33}"]&\stauitilt \tilde{B}\\
			\stautilt B\ar[u,"\text{\eqref{ddual} for \(B\)}"] \ar[r,"\eqref{stau corr 2021-09-07 13:51:09}"']&\stautilt \tilde{B}\ar[u,"\text{\eqref{ddual} for \(\tilde{B}\)}"'].
		\end{tikzcd}
	\end{equation}
	Let \(M\) be an arbitrary support \(\tau\)-tilting \(B\)-module  and \(P\) a projective \(B\)-module satisfying that \(\Hom_B(P,M)=0\) and that \(|M|+|P|=|B|\).
	Then the corresponding support \(\tau^{-1}\)-tilting module of \(M\) in the bijection \eqref{ddual} is \(\tau M\oplus P\).
	Let \(A\) be the direct product of the all blocks of \(k\tilde{G}\) covering \(B\).
	By \cref{induc dsmmand number 2021-09-07 10:26:20} and \cref{counting 2021-09-10 17:42:55}, we get \(|\induc_G^{\tilde{G}} M|+|\induc_G^{\tilde{G}} P|=|A|\).
	Therefore, the corresponding support \(\tau^{-1}\)-tilting module of \(\induc_G^{\tilde{G}} M\) with respect to \eqref{ddual} is \(\tau \induc_G^{\tilde{G}}M\oplus \induc_G^{\tilde{G}} P\), which is isomorphic to the induced module \(\induc_G^{\tilde{G}}(\tau M\oplus P)\) by \cref{induction functor projective cover 2021-09-07 12:06:47}.
\end{proof}
\begin{corollary}\label{MT smcsbrick 2021-10-03 21:01:38}
	With the same assumptions in \cref{MT 2021-09-07 13:50:56}, we get the injective map
	\begin{equation}\label{extending smc 2021-10-02 17:53:01}
		\begin{tikzcd}
			\twosmc B\ar[r]&\twosmc \tilde{B}
		\end{tikzcd}
	\end{equation}
	which makes the following diagram commutative:
	\begin{equation}
		\begin{tikzcd}
			\frsbrick B \ar[rrr,"\eqref{sbrick corr 2021-09-10 18:10:41}"]\ar[dr,leftarrow,"\text{\eqref{asabjfr} for \(B\)}"]\ar[dd,leftarrow,"\text{\eqref{dualsbrick} for \(B\)}"']& & &\frsbrick \tilde{B} \ar[dl,leftarrow,"\text{\eqref{asabjfr} for \(\tilde{B}\)}"'] \ar[dd,leftarrow,"\text{\eqref{dualsbrick} for \(\tilde{B}\)}"]\\
			&\twosmc B \ar[ld,"\text{\eqref{asaibjfl} for \(B\)}"]\ar[r,dashrightarrow,"\eqref{extending smc 2021-10-02 17:53:01}"]& \twosmc \tilde{B}\ar[rd,"\text{\eqref{asaibjfl} for \(\tilde{B}\)}"'] &
			\\
			\flsbrick B\ar[rrr,"\eqref{sbrick corr 2021-09-10 09:27:09}"']&&&\flsbrick \tilde{B}.
		\end{tikzcd}
	\end{equation}
\end{corollary}

The following theorem summaries \cref{MT 2021-09-07 13:50:56}, \cref{MT sbrick 2021-09-10 09:27:04} and \cref{MT smcsbrick 2021-10-03 21:01:38}.
\begin{theorem}\label{big com. diag. 2021-10-06 10:05:34}
	Let \(\tilde{G}\) be a finite group, \(G\) a normal subgroup of \(\tilde{G}\), \(B\) a block of \(kG\) and \(\tilde{B}\) a block of \(k\tilde{G}\) covering \(B\).
	We assume that the following conditions hold:
	\begin{enumerate}
		\item Any left finite bricks in \(B\lmod\) is \(\inertiagp_{\tilde G}(B)\)-stable.
		\item \(H^2(\inertiagp_{\tilde{G}}(B)/G,k^\times)=1\).
		\item \(k[\inertiagp_{\tilde{G}}(B)/G]\) is basic as a \(k\)-algebra.
	\end{enumerate}
	Then the following diagram is commutative:
	\begin{equation}
		\begin{tikzcd}
			&\twosmc B\ar[dd,"\text{\eqref{asaibjfl} for \(B\)}"'very near end]  \ar[rr,"\text{\eqref{extending smc 2021-10-02 17:53:01}}"]&&\twosmc \tilde{B}\\
			\twotilt B\ar[ru,"\text{\eqref{2KYcorr} for B}"description] \ar[rr,"\text{\eqref{twotilt corr 2021-09-10 17:53:30}}"near end,crossing over]&&\twotilt \tilde{B}\ar[ru,"\text{\eqref{2KYcorr} for \(\tilde{B}\)}"description] \\
			&\flsbrick B\ar[rr,"\eqref{sbrick corr 2021-09-10 09:27:09}"very near start]&&\flsbrick \tilde{B}\ar[uu,leftarrow,"\text{\eqref{asaibjfl} for \(\tilde{B}\)}"'].\\
			\stautilt B\ar[ru,"\text{\eqref{asaicorresp} for \(B\)}"description] \ar[uu,"\text{\eqref{2siltcorr} for \(B\)}"] \ar[rr,"\text{\eqref{stau corr 2021-09-07 13:51:09}}"'] &&\stautilt \tilde{B} \ar[ru,"\text{\eqref{asaicorresp} for \(\tilde{B}\)}"description] \ar[uu,"\text{\eqref{2siltcorr} for \(\tilde{B}\)}"'very near end,crossing over]
		\end{tikzcd}
	\end{equation}
\end{theorem}

\subsection{Some applications of main theorems}
We give some applications and examples of main theorems in this section.
The following arguments provide sufficient conditions for the assumptions of our main results.
\begin{proposition}[{\cite[Lemma 3.5.4, Theorem 3.5.11]{MR998775}, \cite[Theorem 5.3, Theorem 5.9]{MR1200878}}]\label{assumption ex 2021-11-23 16:12:58}
	Let \(\tilde{G}\) be a finite group and \(G\) a normal subgroup of \(\tilde{G}\). Then we have the following:
	\begin{enumerate}
		\item If \(\tilde{G}/G\) is a \(p\)-group, then we have that \(H^2(\tilde{G}/G,k^{\times})=1\) and that the group algebra \(k[\tilde{G}/G]\) is basic.\label{item condition p 2022-09-05 14:54:36}
		\item If \(\tilde{G}/G\) is cyclic, then we have that \(H^2(\tilde{G}/G,k^{\times})=1\) and that the group algebra \(k[\tilde{G}/G]\) is basic.\label{item condition cyclic 2022-09-05 14:54:30}
		\item If \(\tilde{G}/G\) is isomorphic to the dihedral group \(D_{2p}\) of order \(2p\), then we have that \(H^2(\tilde{G}/G,k^{\times})=1\) and that the group algebra \(k[\tilde{G}/G]\) is basic.\label{item condition dihedral 2022-09-05 14:55:12}
		\item If \(\tilde{G}/G\) is abelian, then we have that the group algebra \(k[\tilde{G}/G]\) is basic.
	\end{enumerate}
\end{proposition}

The following argument describes the Hasse quiver of support \(\tau\)-tilting modules labeled by bricks and will be used to prove the refinement of the main result of \cite{MR4243358}.
\begin{proposition}\label{label brick 2021-10-03 22:08:25}
	With the same assumptions in \cref{MT 2021-09-07 13:50:56},
	let \(M\) be a support \(\tau\)-tilting \(B\)-module and \(M'\) a support \(\tau\)-tilting left mutation of \(M\) and \(S\) the brick in \(B\lmod\) corresponding to the left mutation \(M'\) in the bijection \eqref{labeling brick 2021-10-13 08:59:12} for \(M\).
	If \(\tilde{B}\induc_G^{\tilde{G}}M'\) is a support \(\tau\)-tilting left mutation of \(\tilde{B}\induc_G^{\tilde{G}}M\), then there exists only one sort of extending \(k\tilde{G}\)-module \(\tilde{S}\) of \(S\) lying in \(\tilde{B}\) up to isomorphisms.
	In particular, the brick labeling the arrow from \(\tilde{B}\induc_G^{\tilde{G}}M\) to \(\tilde{B}\induc_G^{\tilde{G}}M'\) in \(\mathcal{H}(\stautilt \tilde{B})\) is isomorphic to \(\tilde{S}\).
\end{proposition}
\begin{proof}
	Let \(N\) and \(N'\) be the support \(\tau^{-1}\)-tilting \(B\)-modules corresponding to \(M\) and \(M'\) respectively in the bijection \eqref{ddual}.
	By \cref{labeling properties 2021-11-14 17:10:59}, we have \(S \in \Fac M \cap \Sub N'\).
	Therefore, we get \(\tilde{B}\induc_G^{\tilde{G}}S \in \Fac \tilde{B}\induc_G^{\tilde{G}} M \cap \Sub \tilde{B}\induc_G^{\tilde{G}} N'\).
	By the assumption, the subcategory \(\Fac \tilde{B}\induc_G^{\tilde{G}} M \cap \Sub \tilde{B}\induc_G^{\tilde{G}} N'\) of \(\tilde{B}\lmod\) contain a unique brick in \(\tilde{B}\lmod\) by \cref{labeling properties 2021-11-14 17:10:59}.
	Hence, we get that there exists only one sort of extending \(k\tilde{G}\)-module of \(S\) lying in \(\tilde{B}\) up to isomorphisms.
	Therefore, the extending \(k\tilde{G}\)-module of \(S\) lying in \(\tilde{B}\) is isomorphic to \(\tilde{S}\).
\end{proof}

Now we state some consequences of our main results stated in \cref{subsec main results 2022-09-07 16:40:50}.
The following result is a more precise result than \cite[Theorem 4.2]{MR4243358}.
\begin{theorem}\label{p-extension 2021-10-06 10:06:35}
	Let \(\tilde{G}\) be a finite group, \(G\) a normal subgroup of \(\tilde{G}\) such that the quotient group \(\tilde{G}/G\) is a \(p\)-group, \(\tilde{B}\) a block of \(k\tilde{G}\) and \(B\) a block of \(kG\) covered by \(\tilde{B}\).
	Assume that any left finite brick in \(B\lmod\) is \(\inertiagp_{\tilde{G}}(B)\)-stable.
	Then we have the following:
	\begin{enumerate}
		\item The maps \eqref{stau corr 2021-09-07 13:51:09} and \eqref{sbrick corr 2021-09-10 09:27:09} induce the embedding of quiver with labels:
		      \begin{equation}
			      \begin{tikzcd}
				      \mathcal{H}(\stautilt B)\ar[r]& \mathcal{H}(\stautilt \tilde{B}).
			      \end{tikzcd}
		      \end{equation}
		      Moreover, any connected component of \(\mathcal{H}(\stautilt B)\) is embedded as a connected component of \(\mathcal{H}(\stautilt \tilde{B})\).
		\item If \(B\) is a support \(\tau\)-tilting finite block, then the map \eqref{stau corr 2021-09-07 13:51:09} is an isomorphism from \(\stautilt B\) to \(\stautilt \tilde{B}\) as partially ordered sets. Moreover, all the maps appearing in the commutative diagram in \cref{big com. diag. 2021-10-06 10:05:34} are bijective. In particular, the maps \eqref{stau corr 2021-09-07 13:51:09} and \eqref{sbrick corr 2021-09-10 09:27:09} induce an isomorphism from \(\mathcal{H}(\stautilt B)\) to \(\mathcal{H}(\stautilt \tilde{B})\) as labeled quivers.
	\end{enumerate}
\end{theorem}
\begin{proof}
	By \cref{p-power index covering uniqueness}, we get \(\induc_G^{\tilde{G}} U=\tilde{B}\induc_G^{\tilde{G}} U\) for any \(B\)-module \(U\).
	By the same arguments of \cite[Theorem 4.2]{MR4243358},
	the induction functor \(\induc_G^{\tilde{G}}\) preserves support \(\tau\)-tilting left mutations.
	Hence, we get conclusions by \cref{label brick 2021-10-03 22:08:25}.
\end{proof}
\begin{corollary}\label{bijectice 2022-08-24 21:47:43}
	Let \(\tilde{G}\) be a finite group and \(G\) be a normal subgroup with cyclic Sylow \(p\)-subgroup such that the quotient group \(\tilde{G}/G\) is a \(p\)-group.
	Then the induction functor \(\induc_G^{\tilde{G}}\) induces the following isomorphism as partially ordered sets:
	\begin{equation}\label{kg stau 2021-10-08 14:56:34}
		\begin{tikzcd}[row sep=1pt]
			\stautilt kG\ar[r]& \stautilt k\tilde{G}\\M \ar[r,mapsto] &\induc_G^{\tilde{G}} M.
		\end{tikzcd}
	\end{equation}
\end{corollary}
\begin{proof}
	Since any defect group of a block of \(kG\) is contained in a Sylow \(p\)-subgroup of \(G\), any block has a cyclic defect group.
	Hence, any block of \(kG\) is \(\tau\)-tilting finite.
	Thus, the conclusion follows from \cref{p-extension 2021-10-06 10:06:35} for all blocks of \(kG\).
\end{proof}
\begin{proposition}\label{cyclic quot 2022-08-28 22:53:27}
	Let \(G\) be a normal subgroup of a finite group \(\tilde{G}\), \(\tilde{B}\) a block of \(k\tilde{G}\) and \(B\) a cyclic defect block of \(kG\) covered by \(\tilde{B}\) satisfying one of the following conditions:
	\begin{enumerate}
		\item There is an \(\inertiagp_{\tilde{G}}(B)\)-stable simple \(B\)-module \(S\) whose corresponding edge is a terminal edge of the Brauer tree of \(B\).\label{item end edge 2022-08-28 17:31:12}
		\item There is a simple \(B\)-module \(S\) whose corresponding edge of the Brauer tree of \(B\) is a terminal edge and the dimension of \(S\) whose dimension of \(S\) is distinct to that of any other simple \(B\)-module.\label{item end edge dim 2022-09-05 13:46:08}
		\item Any two simple \(B\)-modules have distinct dimensions.\label{item diff 2022-08-28 17:31:26}
	\end{enumerate}
	If the quotient group \(\tilde{G}/G\) is a cyclic group or isomorphic to the dihedral group \(D_{2p}\) of order \(2p\), then \cref{MT 2021-09-05 21:34:22,MT 2022-08-24 20:58:25,MT sbrick 2021-09-06 11:12:51} hold.
\end{proposition}
\begin{proof}
	We can assume that \(\inertiagp_{\tilde{G}}(B)=\tilde{G}\) by \cref{Morita equivalence covering block}.
	We enough to show that the three conditions in \cref{MT 2021-09-07 13:50:56} are satisfied in the situations \ref{item end edge 2022-08-28 17:31:12},\ref{item end edge dim 2022-09-05 13:46:08} and \ref{item diff 2022-08-28 17:31:26}.
	By \cref{assumption ex 2021-11-23 16:12:58}, we have that the second and third conditions are satisfied in our situation.
	Assume that the block \(B\) of \(kG\) satisfies the condition \ref{item end edge 2022-08-28 17:31:12} and let \(S\) be an \(\inertiagp_{\tilde{G}}(B)\)-stable simple \(B\)-module whose corresponding edge is a terminal edge of the Brauer tree of \(B\).
	Then, since there exists a unique simple \(B\)-module \(T\) such that \(\Ext_B^1(S,T)\cong k\) and that \(\Ext_B^1(S,T')= 0\) for any distinct simple \(B\)-module \(T'\) to \(T\), we have that
	\begin{equation}
		\Ext_B^1(S,xT)\cong \Ext_B^1(xS,xT)\cong \Ext_B^1(S,T)\cong k.
	\end{equation}
	Hence, we have \(xT\cong T\) as \(B\)-modules for any \(x\in \inertiagp_{\tilde{G}}(B)\) by the uniqueness of \(T\) again.
	Also, since there exists a unique simple \(B\)-module \(U\) distinct to \(S\) such that \(\Ext_B^1(T,U)\cong k\) and \(\Ext_B^1(T,U')=0\) for any distinct simple \(B\)-module \(U'\) to \(U\) and \(S\), we have that
	\begin{equation}
		\Ext_B^1(T,xU)\cong\Ext_B^1(xT,xU)\cong\Ext_B^1(T,U)\cong k,
	\end{equation}
	which implies that \(xU\cong U\) as \(B\)-modules for any \(x\in \inertiagp_{\tilde{G}}(B)\).
	By repeating this argument, we have that any simple \(B\)-module is \(\inertiagp_{\tilde{G}}(B)\)-stable.
	Therefore, we have that any \(B\)-module is \(\inertiagp_{\tilde{G}}(B)\)-stable by \cref{Cyclicdefect p-power index 2022-08-28 21:45:58} \ref{2021-12-06 18:35:10}.
	In particular, the block \(B\) satisfies the first condition in \cref{MT 2021-09-07 13:50:56}.

	Next, assume that the block \(B\) of \(kG\) satisfies the condition \ref{item end edge dim 2022-09-05 13:46:08}.
	Then a simple \(B\)-module \(S\) whose corresponding edge is a terminal edge of the Brauer tree of \(B\) is \(\inertiagp_{\tilde{G}}(B)\)-stable because \(xS\) is a simple \(B\)-module with the same dimension as \(S\) for any \(x\in \inertiagp_{\tilde{G}}(B)\).
	Therefore, by \ref{item end edge 2022-08-28 17:31:12}, the block \(B\) satisfies the first condition in \cref{MT 2021-09-07 13:50:56}.
	The statement for \ref{item diff 2022-08-28 17:31:26} follows from that for \ref{item end edge dim 2022-09-05 13:46:08} immediately
\end{proof}
\begin{corollary}\label{cyclic quot principal 2022-08-28 22:59:59}
	Let \(G\) be a finite group with a cyclic Sylow \(p\)-group and \(\tilde{G}\) a finite group having \(G\) as a normal subgroup.
	If the quotient group \(\tilde{G}/G\) is a cyclic group,
	Then the principal block \(B_0(kG)\) satisfies the three conditions in \cref{MT 2021-09-07 13:50:56}.
	Therefore, \cref{MT 2021-09-05 21:34:22,MT 2022-08-24 20:58:25,MT sbrick 2021-09-06 11:12:51} hold for the principal blocks \(B_0(kG)\) and \(B_0(k\tilde{G})\).
\end{corollary}
\begin{proof}
	The trivial \(kG\)-module \(k_G\) is \(\tilde{G}\)-stable.
	Moreover, the trivial \(kG\)-module corresponds to the terminal edge in the Brauer tree of the principal block \(B_0(kG)\) (for example, see \cite[section 1.1]{zbMATH00044302}).
	Hence, it concludes the proof by \cref{cyclic quot 2022-08-28 22:53:27}.
\end{proof}

\begin{example}\label{example 2022-08-24 21:48:34}
	Let \(G:=\mathfrak{A}_5\) be the alternating group of degree \(5\), \(\tilde{G}:=\mathfrak{S}_5\) the symmetric group of degree \(5\) and \(P\) a Sylow \(5\)-subgroup of \(G\).
	Since \(P\) is cyclic and the centralizer \(C_{\tilde{G}}(P)=P\) is contained in \(G\),
	the only block covering the principal block \(B_0(kG)\) is the principal block \(B_0(k\tilde{G})\) of \(k\tilde{G}\) by \cref{cover principal 2022-09-05 14:33:21} and \cite[Theorem 4.15.1 (5)]{MR860771}.
	Also, the number \(\#\stautilt B_0(k\tilde{G})\) of support \(\tau\)-tilting \(B_0(k\tilde{G})\)-modules is equal to \(\binom{8}{4}=70\) by \cite{aoki2019classifying,ASASHIBA2020119}. So it is difficult to classify support \(\tau\)-tilting \(B_0(k\tilde{G})\)-modules.
	On the other hand, the classification of support \(\tau\)-tilting \(B_0(kG)\)-modules is easy because the number \(\#\stautilt{B_0(kG)}\) of support \(\tau\)-tilting \(B_0(kG)\)-modules is equals to \(\binom{4}{2}=6\) by \cite{aoki2019classifying} or \cite{ASASHIBA2020119} again.
	Hence, we can easily construct six support \(\tau\)-tilting \(B_0(k\tilde{G})\)-modules, six semibricks over \(B_0(k\tilde{G})\), six two-term tilting complexes in \(K^b(B_0(k\tilde{G})\lproj)\) and six two-term simple-minded collections in \(D^b(B_0(k\tilde{G})\lmod)\) from \(\stautilt{B_0(kG)}\) by using our main theorems.
\end{example}
\begin{example}\label{example 2022-08-24 21:48:53}
	Let \(k\) be an algebraically closed field of characteristic \(p\), \(G\) an arbitrary finite group and \(H\) a finite group satisfying that \(H^2(H,k^{\times})=1\) and that \(kG_2\) is basic as a \(k\)-algebra.
	For example, we can take a \(p\)-group, a cyclic group or the dihedral group \(D_{2p}\) of order \(2p\) by \cref{assumption ex 2021-11-23 16:12:58}.
	Also, it is clear that \(M\cong xM\) for any \(kG\)-module \(M\) and for any \(x\in G\times H\).
	Hence, for the direct product group \(G\times H\), we can apply our main theorem.
	The induction functor \(\induc_{G}^{G\times H}\cong kH\otimes_k \bullet\) induces the injective map \(\stautilt kG\rightarrow \stautilt k[G\times H]\).
\end{example}
\subsection*{Acknowledgements}
The authors are grateful to Naoko~Kunugi for giving valuable comments.
The first author would like to thank Yuta~Katayama for useful advice about the group cohomology of the dihedral groups.

\Addresses
\end{document}